\documentclass[a4paper,11pt]{article}
\usepackage{graphicx}
\usepackage[T1]{fontenc}
\usepackage[utf8]{inputenc}
\usepackage[english]{babel}
\usepackage{amsthm,amssymb,amsfonts}
\usepackage{amsmath}
\usepackage{lmodern}
\usepackage{subcaption}
\usepackage{geometry}
 \usepackage[boxed]{algorithm2e}

\usepackage{framed}
 \usepackage{color}
\definecolor{darkgreen}{rgb}{0.00,0.5,0.00}
\usepackage{hyperref}
\hypersetup{
	colorlinks=true,        
	linkcolor=blue,         
	citecolor=darkgreen,         
	urlcolor=blue           
}
\usepackage{lscape} 
\usepackage{cite}
\usepackage{enumerate}

\newtheorem{theorem}{Theorem}
\theoremstyle{definition}
\newtheorem{lemma}{Lemma}
\newtheorem{proposition}{Proposition}
\newtheorem{remark}{Remark}
\newtheorem{corollary}{Corollary}
\newtheorem{assumption}{Assumption}

\newcommand{\R}{\mathbb R}

\newcommand{\M}{\mathcal{M}}

\newcommand{\y}{\bar y}

\newcommand{\FBE}{\varphi_{\gamma,f,g}^{\rm FB}}
\newcommand{\RFB}{Y^{\rm FB}_{f,g}}

\newcommand{\RDY}{Y^{\rm DY}_{f,h,g}}
\newcommand{\DYE}{\varphi_{\gamma}^{\rm DY}}

\DeclareMathOperator{\inte}{int}
\DeclareMathOperator{\prox}{prox}

\DeclareMathOperator{\dom}{dom}

\DeclareMathOperator*{\argmin}{argmin}

\title{Splitting Methods for Nonconvex Optimisation: Convergence and Saddle Point Avoidance Through Shadow Sequences}
\author{Felipe Atenas\footnote{{\scriptsize Centro de Modelamiento Matem\'{a}tico (CNRS IRL2807), Universidad de Chile, Santiago, Chile, \href{mailto:fatenas@cmm.uchile.cl}{fatenas@cmm.uchile.cl}}}}
\date{\today}

\begin{document}
\maketitle

\begin{abstract}
In nonconvex optimisation, commonly used methods are usually shown to converge to stationary points, while guarantees of convergence to local minimisers have comparatively received less attention. To close this gap, we analyse the convergence behaviour of the Davis--Yin three-operator splitting method. More precisely, for structured weakly convex semialgebraic optimisation problems, we establish convergence to local minimisers, almost surely. We propose a novel perspective based on the so-called shadow sequence, deviating from the traditional strategy that uses fixed-point sequences. We perform numerical experiments that suggest that the Davis--Yin method is more robust with respect to the selection of parameters than the forward-backward method, even in  settings where the latter may seem to be the obvious choice for solution method.
\end{abstract}

\medskip


\section{Introduction}

We consider structured optimisation problems of the form \begin{equation} \label{DY-problem}
    \min_{x \in \R^d} \varphi(x) := f(x) + g(x) + h(x),
\end{equation} where $f: \R^d \to \R$ and $g: \R^d \to \R \cup \{+\infty\}$ are \emph{prox-friendly} functions, that is, their proximal operators (see \eqref{eq:prox}) are readily available or easy to compute, and $h: \R^d \to \R$ is a smooth function. Problems with this type of structure appear in several applications in statistics, signal processing, and machine learning \cite{peng2016coordinate, parikh2014proximal,boyd2011distributed}. 

A solution method that exploits the sum structure of the objective function in \eqref{DY-problem} is the Davis--Yin (DY) three-operator splitting method \cite{davis2017three}, whose iterates are defined as follows. Given  an initial point $z^0 \in \R^d$, a stepsize parameter $\gamma >0$, a relaxation parameter $\lambda > 0$, and an iteration counter $k\geq1$, the $k$-th iteration of the DY method is given by
\begin{equation} \label{DY}
    \left\{\begin{array}{rcl}
        x^k & = & \prox_{\gamma f}(z^k)  \\
        y^k & = & \prox_{\gamma g}(2x^k-z^k - \gamma \nabla h(x^k)) \\
        z^{k+1} & = &  z^k + \lambda (y^k - x^k).
    \end{array}\right.
\end{equation} The iterative scheme in \eqref{DY} reduces to two well-known two-operator splitting methods \cite{douglas1956numerical,lions1979splitting}: when $f =0$, it corresponds to the proximal gradient or \emph{forward-backward} (FB) method, and when $h = 0$, it is the \emph{Douglas--Rachford} (DR) method.  The DY method was originally introduced in \cite{davis2017three} to solve convex optimisation problems in the form of \eqref{DY-problem}, under the assumption that both $f$ and $g$ are convex, and $h$ is smooth and convex with cocoercive gradient. Therein, it is shown that the fixed-point sequence $(z_{k})_{k\geq 1}$ converges to a fixed point of the underlying iteration operator, and that the sequences $(x^{k})_{k\geq 1}$ and $(y^{k})_{k\geq 1}$ converge to a minimiser of $f+g+h$. The behaviour of splitting methods for convex optimisation problems is well-documented (see, e.g., \cite{MR4584772} and references therein), whereas the literature in nonconvex settings is comparatively more scarce. In particular, for the DY method, convergence guarantees for nonconvex optimisation problems are limited to finding stationary points \cite{bian2021three,alcantara2024four}. Our main interest in this work is investigating the reach of the DY method for nonconvex optimisation problems. Our goal is establishing conditions under which the DY method converges to local minimisers. 

Pivotal in our analysis is the so-called \emph{Davis--Yin envelope} defined in \cite{liu2019envelope},  originally tailored to have the fixed-point sequence $(z^k)_{k\geq1}$ as argument. Instead, the envelope in this work is defined in such a way that its argument is associated with  the sequence $(x^k)_{k\geq1}$ in \eqref{DY}, the so-called \emph{shadow sequence}. Namely, \emph{our} Davis--Yin envelope (DYE) is defined for all $x \in \R^d$ as \begin{equation}
    \label{DY-env-x}
    \varphi_\gamma^{\rm DY}(x)  :=  \inf_{y \in \R^d} \Big\{ f(x) + \langle \nabla f(x), y - x\rangle  + \: h(x) + \langle \nabla h(x), y -x\rangle   \\  + \: g(y) + \frac{1}{2\gamma}\| y -  x \|^2 \Big\}.
\end{equation} This seemingly minor change unlocks a new perspective on the study of splitting methods 
based on the shadow sequence, an approach that has been largely unexplored. More importantly, this change of perspective enables the investigation of how the DY method avoids solutions of ``poor quality'', that is, (strict) saddle points, when nonsmoothness appears in a structured manner.

\noindent\emph{Related work.} In \cite{bian2021three,alcantara2024four}, the authors established convergence of the DY method in a nonconvex setting by means of suitable \emph{merit functions} that satisfy a sufficient descent condition alongside iterations, where $f$ and $h$ are smooth with Lipschitz continuous gradients, and $g$ is prox-friendly nonconvex. Under a similar lens, another three-operator splitting method for nonconvex optimisation problems is examined in \cite{alcantara2025relaxed}. For two-operator splitting methods in nonconvex optimisation, \cite{li2016douglas,themelis2020douglas} investigated convergence properties of the DR method under the same assumptions above, namely, $f$ smooth with Lipschitz continuous gradient, and $g$ nonconvex, possibly nonsmooth.

In the nonconvex smooth optimisation case, first-order methods for the minimisation of $C^2$ functions avoid stationary points with negative curvature when randomly initialised \cite{lee2016gradient, panageas2016gradient}. Moreover, in \cite{davis2022proximal}, this result is extended to methods of proximal type for nonconvex problems with structured nonsmoothness, namely, under the assumption of the \emph{strict saddle property}. The methods included in the latter work are the proximal point algorithm, the FB and the prox-linear methods. The heart of the analysis therein is the \emph{Centre-Stable manifold theorem} \cite[Theorem III.7]{shub2013global}. The authors of \cite{liu2019envelope} analysed  \eqref{DY} as a variable metric gradient method for the Davis--Yin envelope, and obtained conditions for the FB and the DR methods to avoid \emph{strict saddle points} with probability one in the fully smooth case, namely, assuming in addition that all functions involved are $C^2$. In this work, we show a saddle point avoidance result for the DY method dropping the fully smooth assumption, replacing the smoothness assumption on $g$ with the hypothesis used in \cite{davis2022proximal}. 

In \cite{csetnek2019shadow} the shadow Douglas--Rachford method is introduced for monotone inclusions, and thus applicable to convex optimisation. For nonconvex optimisation, in \cite{themelis2020douglas,bian2021three,alcantara2024four}, the authors examine splitting methods using envelopes that take the fixed-point sequence $(z^k)_{k\geq1}$ as input. The relevance of the latter sequence comes from the convex optimisation case (or, more generally, for the maximally monotone case for operators), as splitting methods are known to be equivalent to fixed-point iterations of an appropriate nonexpansive operator, see \cite{eckstein1992douglas,malitsky2023resolvent,atenas2025relocated} and references therein. 
Although useful, the fixed-point sequence does not seem to reveal new insights on splitting methods for nonconvex optimisation, and thus we propose to centre the analysis around the shadow sequence. Another fact that supports our proposed approach is that, in both convex and nonconvex cases, the customary argument is the following: establish (subsequential) convergence of the fixed-point sequence $(z^k)_{k\geq1}$, from where solutions of the problem in \eqref{DY-problem} are recovered by applying $\prox_{\gamma f}$ to the limits of $(z^k)_{k\geq1}$, which are in turn limits of the shadow sequence $(x^k)_{k\geq1}$. This line of reasoning uses  $(z^k)_{k\geq1}$ as an artefact that facilitates the analysis, but it does not directly provide solutions to the optimisation problem. In this work, we skip the fixed-point sequence $(z^k)_{k\geq1}$, and directly inspect the shadow sequence $(x^k)_{k\geq1}$.



\noindent\emph{Novelty.} Our contribution can be summarised as follows. We show saddle point avoidance for the DY method when the nonsmooth function $g$ in \eqref{DY-problem} admits a $C^2$-smooth active manifold, extending the work \cite{liu2019envelope} where similar results were obtained separately for the FB and the DR methods in the fully smooth case, and generalising the structured nonsmooth setting for proximal methods in \cite{davis2022proximal}. We also provide a new perspective for establishing convergence (to stationary points) of proximal splitting methods in nonconvex optimisation centred on the shadow sequence. Finally, we provide numerical evidence of parameter robustness of the DY method over the FB method. This motivates the application of the DY method even when $f+h$ is smooth in problem \eqref{DY-problem}. As far as the author knows, tests of this type are nonexistent in the literature.

\noindent\emph{Structure of the paper.} The rest of this paper is organised as follows. In Section~\ref{s:prelim}, we present some background needed for our analysis. 
In Section~\ref{s:continuity-epi}, we briefly discuss envelopes, the merit functions used to study proximal splitting methods in nonconvex settings. Section~\ref{s:damped-DY} is dedicated to the convergence analysis of a version of the DY method with damping on the shadow sequence. We then show in Section~\ref{s:saddle-avoidance} that the damped DY method converges to local minimisers almost surely under appropriate structural properties of the objective function. Section~\ref{s:numerical} compares the DY and the FB methods in nonconvex optimisation, and illustrates our results on saddle point avoidance. In Section~\ref{s:conclusion}, we conclude with some final remarks and future research directions.

\section{Notation and Preliminaries} \label{s:prelim}

A function $\phi: \R^d \to \R \cup \{+\infty\}$ is said to be proper when its domain $\dom(\phi) := \{ x \in \R^d: \phi(x) < +\infty\} \neq \emptyset$ is nonempty. We say that $\phi$ is level-bounded if for any $\alpha \in \R$, the set ${\rm lev}_\phi(\alpha) := \{x\in \R^d: \phi(x) \leq \alpha\}$ is bounded. A function $\phi$ is locally Lipschitz continuous if for all $\bar{x} \in \dom(\phi)$, there exists a neighbourhood $U$ of $\bar{x}$, and a constant $L = L(U)>0$, such that for all $x, y \in U$, $|\phi(x)-\phi(y)| \le L\|x-y\|$. We say $\phi$ is (globally)  Lipschitz continuous if the latter estimate holds for $U = \R^d$. 
Given three functions $\phi_1$, $\phi_2$, $\phi_3$, defined on spaces for which $(\phi_1 \circ \phi_2) + \phi_3$ is well-defined, and such that $\phi_i$ is Lipschitz continuous with constant $L_i$, $i=1,2,3$, then $(\phi_1 \circ \phi_2) + \phi_3$ is Lipschitz continuous with constant $L_1L_2 + L_3$. A function $\phi$ is $\rho$-weakly convex ($\rho\geq0$) whenever $\phi(\cdot) + \frac{\rho}{2}\|\cdot\|^2$ is convex. For a differentiable function $\phi$, we say $\phi$ is $L_\phi$-smooth if its gradient $\nabla \phi$ is Lipschitz continuous with parameter $L_\phi \geq 0$. In particular, such a function $\phi$ is $\rho$-weakly convex, where $\rho$ is bounded above by the smallest Lipschitz constant of $\nabla \phi$ \cite[Proposition 2.4]{atenas2023unified}.  If $\phi$ is a $L_\phi$-smooth $C^2$ function, then for all $x \in \R^d$, $\|\nabla^2 \phi(x)\|_{\rm op} \leq L_{\phi}$, where $\|\cdot\|_{\rm op} $ denotes the operator norm, and $I \pm \gamma \nabla^2\phi(x)$ is positive definite for $\gamma \in (0, L_\phi^{-1})$.

\subsection{Variational Analysis} Given a point $x \in \R^d$ and a proper function $\phi: \R^d \to \R \cup \{+\infty\}$, the proximal operator of $\phi$ at $x$ with prox-parameter $\gamma > 0$ is defined as \begin{equation} \label{eq:prox}
    \prox_{\gamma \phi}(x) :=  \argmin\limits_{y \in \R^d} \left\{ \phi(y) + \frac{1}{2\gamma}\| y - x\|^2 \right\}.
\end{equation} The optimal value of this minimisation problem defines the Moreau envelope at $x$, \begin{equation*}
    \phi^M_\gamma (x) := \inf_{y \in \R^d}\left\{ \phi(y) + \frac{1}{2\gamma}\|y-x\|^2\right\}.
\end{equation*} 
If $\phi$ is $\rho$-weakly convex, 
then for any $\gamma \in (0, \rho^{-1})$, $\prox_{\gamma \phi}$ is single-valued and Lipschitz continuous with constant $\frac{1}{1-\gamma\rho}$ \cite[Proposition 12.19]{rockafellar2009variational}, and $\phi^M_\gamma$  is continuously differentiable with gradient given by \begin{equation}\label{eq:moreau-grad}\nabla \phi_{\gamma}^M(x) = \frac{1}{\gamma}\big(x - \prox_{\gamma \phi}(x)\big),\end{equation}such that $\nabla \phi_{\gamma}^M$ is Lipschitz continuous with constant $\max\{\gamma^{-1},\frac{\rho}{1-\rho\gamma}\}$ \cite[Corollary 3.4]{hoheiselproximal}. If, in addition, $\phi$ is twice continuously differentiable, then $\prox_{\gamma \phi} $ is continuously differentiable \cite[Lemma 4.2]{liu2019envelope}.

For a proper lower semicontinuous function $\phi: \R^d \to \R \cup \{+\infty\}$, the (limiting) subdifferential of $\phi$ at $\bar{x} \in \dom(\phi)$ \cite[Definition 8.3]{rockafellar2009variational} is given by 
\begin{equation*}
    \partial \phi(\bar{x}) = \limsup_{\substack{x \to \bar{x}\\\phi(x) \to \phi(\bar{x})}}\left\{ s \in \R^d: \liminf_{\substack{y \to x\\y\neq x}} \frac{\phi(y) -\phi(x) - \left<s,y-x \right>}{\|y-x\|} \geq 0 \right\}.
\end{equation*} When $\phi$ is convex, $\partial \phi$ coincides with the subdifferential of convex analysis. In the latter case, the subdifferential is known to be a (maximally) monotone operator \cite[Theorems~12.17~\&~12.25]{rockafellar2009variational}. A point $x \in \R^d$ such that $ 0 \in \partial \phi (x)$ is said to be a stationary point of $\phi$. For the problem in \eqref{DY-problem}, $x$ is a stationary point of $\varphi$ if $ 0 \in \partial g(x) + \nabla (f+h)(x)$ in view of the subdifferential sum rule \cite[Exercise 8.8(c)]{rockafellar2009variational}. We say $\phi^\star \in \R$ is a stationary value of the function $\phi$ if there exists a stationary point $x^\star$ of $\phi$ such that $\phi(x^\star) = \phi^\star$.

\subsection{Active Manifolds}

Given a map $T: \R^d \to \R^d$, we say $T$ is a \emph{lipeomorphism} if $T$  is  Lipschitz continuous, and its inverse $T^{-1}$ exists and is Lipschitz continuous as well. A point $\bar z$ is a fixed point of $T$ if $\bar{z} = T(\bar z)$.  If  $T$ is a $C^1$ map around some fixed point $\bar z$, we say $\bar z$ is an \emph{unstable fixed point} of $T$ if the Jacobian $\nabla T(\bar z)$ has at least one eigenvalue with magnitude strictly greater than one.

We say that a set $\M \subseteq \R^d$ is a $C^2$-\emph{smooth manifold} around $\bar{x} \in \M$ \cite[Definition 2.2]{davis2022proximal}, if there exist an open neighbourhood $U$ of $\bar{x}$, and a $C^2$ function $G$ defined on $\R^d$, such that $\nabla G(\bar{x})$ has full rank, and $\M \cap U = \{x\in\R^d: G(x) = 0\}$. 
The system of equations $G=0$ is called the local defining equations for $\M$ around $\bar{x}$, and ${\mathcal{T}}_{\M}(\bar{x}) := \mbox{ker}\big(\nabla G(\bar{x})\big)$ is the tangent space to $\M$ at $\bar{x}$.

Let $\phi: \R^d \to \R \cup \{+\infty\}$ be a proper lower semicontinuous $\rho$-weakly convex function, $\bar{x} \in\R^d$ be a stationary point of $\phi$, and  $\M \ni \bar{x}$. The set $\M$ is said to be an \emph{active} $C^2$\emph{-smooth manifold} of $\phi$ around $\bar{x}$ \cite[Definition 2.6]{davis2022proximal}, if there exists a neighbourhood $U$ of $\bar{x}$, such that the following two properties hold: (i) (smoothness) $\M \cap U$ is a $C^2$-smooth manifold around $\bar{x}$, such that $\phi$ is a $C^2$ function on $\M \cap U$, and (ii) (sharpness) $\inf\{\|s\|: s \in \partial \phi(x), \: x \in U\setminus\M \}>0$. 
Naturally, these two conditions hold whenever $\phi$ is a $C^2$ function.

We say that a stationary point $\bar{x}$ of a weakly convex function  $\phi: \R^d \to \R \cup \{+\infty\}$ is a \emph{strict saddle} (point) of $\phi$ \cite[Definition 2.7]{davis2022proximal}, if there exists an active $C^2$-smooth manifold $\M$ of $\phi$ at $\bar{x}$, such that for some vector $u \in {\mathcal{T}}_{\M}(\bar{x})$, the parabolic subderivative \cite[Definition 13.59]{rockafellar2009variational} satisfies $d^2 \phi_{\M}(\bar{x})(u) <0$.  A geometric interpretation of strict saddle points is that the function $\phi$ restricted to $\M$ has negative curvature at such points. Moreover, we say $\phi$ satisfies the \emph{strict saddle property}, if any stationary point of $\phi$ is either a strict saddle point or a local minimiser. As shown in \cite[Theorem~2.9]{davis2022proximal}, the strict saddle property is ``generic'', in the sense that it holds for perturbations of $\phi$ drawn from a set of full Lebesgue measure, when $\phi$ is lower semicontinuous weakly convex and semialgebraic. A function is semialgebraic if its graph can be described as the solution of a finite system of polynomial inequalities. The semialgebraic assumption is now state-of-the-art due to the seminal work \cite{attouch2013convergence}. Examples of weakly convex semialgebraic functions are nonconvex regularisers used in statistics, for instance, the \emph{smoothly clipped absolute deviation} (SCAD)  and  the \emph{minimax concave penalty} (MCP) \cite{fan2001variable,Zhang2010NearlyUV}.  

\section{The DYE and Solution Maps} \label{s:continuity-epi}



In this section, we motivate the definition of the DYE \eqref{DY-env-x}, and state some properties of it and of the corresponding solution maps that we will use in our analysis.

From \eqref{eq:moreau-grad}, it directly follows that $\prox_{\gamma \phi} = I -\gamma \nabla \phi_\gamma^M$ when $\phi$ is weakly convex. Hence, the proximal point method \cite{martinet,rockafellar1976monotone}, whose iterates are defined via $x^{k+1} = \prox_{\gamma \phi}(x^k)$ for all $k\geq1$, can be interpreted as the gradient descent method on $\phi_\gamma^M$. For this reason, similar merit functions have been proposed in the literature \cite{themelis2020douglas,liu2019envelope} to analyse splitting methods as methods of descent. In particular,  in \cite{liu2019envelope} the following merit function is introduced: for all $z \in \R^d$,  \begin{multline*}
         V_\gamma(z) := f_{\gamma}^M(z) - \gamma \|\nabla f_{\gamma}^M(z)\|^2 - \gamma \langle \nabla h(\prox_{\gamma f}(z)), \nabla f_{\gamma}^M(z) \rangle + h(\prox_{\gamma f}(z)) \\ - \frac{\gamma}{2}\|\nabla h(\prox_{\gamma f}(z))\|^2 + g_{\gamma}^M\big(z - 2\gamma \nabla f_{\gamma}^M(z) - \gamma \nabla h (\prox_{\gamma f}(z))\big).
    \end{multline*}This definition is inspired by the Douglas--Rachford envelope introduced in \cite{patrinos2014douglas}. One of the advantages of the above formulation is its explicit relation with the Moreau envelope, and thus properties of the latter can be exploited directly. Nevertheless, it lacks interpretability. Here, we reformulate it similarly to the merit functions presented in \cite{themelis2020douglas,alcantara2024four}. After some algebraic manipulations, we obtain\begin{multline} \label{DY-z}
    V_\gamma(z)  =  \inf_{y \in \R^d} \left\{ f(\prox_{\gamma f}(z)) + \langle \nabla f(\prox_{\gamma f}(z)), y -\prox_{\gamma f}(z)\rangle \right. \\ \left. + \: h(\prox_{\gamma f}(z)) + \langle \nabla h(\prox_{\gamma f}(z)), y -\prox_{\gamma f}(z)\rangle \right. \\ \left. + \: g(y) + \frac{1}{2\gamma}\| y -  \prox_{\gamma f}(z) \|^2 \right\}.
\end{multline} The formulation in \eqref{DY-z} suggests that $V_\gamma$ is the value function associated with a model of the objective function in \eqref{DY-problem}, where $f$ and $h$ are replaced by their first-order Taylor approximations. The first-order optimality condition of the problem in the second line of \eqref{DY} reveals that $y = y^k$ solves the problem in \eqref{DY-z} when $z = z^k$. Observe that the argument of the merit function in \eqref{DY-z} is related to the sequence $(z^k)_{k\geq1}$ in \eqref{DY}.  The merit function in \eqref{DY-z} coincides with the one defined in \cite[eq. (3.7)]{alcantara2024four}, when $p = 0$ therein.

In this work, we propose the DYE  introduced in \eqref{DY-env-x}, in such a way that its argument is linked to the sequence $(x^k)_{k\geq1}$ in \eqref{DY}. Observe that $\varphi_\gamma^{\rm DY} \circ \prox_{\gamma f} = V_\gamma $. This seemingly naive change of variables has a direct impact on the study of saddle point avoidance (see Section~\ref{s:saddle-avoidance}), as it explicitly relates the DYE with the merit function known as the \emph{forward-backward envelope}, introduced and extensively studied in \cite{patrinos2013proximal,stella2017proximal,MR3654182}. More precisely, given $x \in \R^d$, the forward-backward envelope \begin{equation} \label{FB-env}
    \FBE(x) := \inf_{y \in \R^d} \left\{ f(x) + \langle \nabla f(x), y -x\rangle + g(y) + \frac{1}{2\gamma}\| y - x\|^2 \right\},
\end{equation}  satisfies \begin{equation} \label{DY-FB-identity-2}
    \DYE(x) = \varphi_{\gamma,f+h,g}^{\rm FB}(x),
\end{equation}   dropping the composition with $\prox_{\gamma f}$ as above. We refer the reader to \cite{MR3845278}, where the envelope is a key element for the convergence analysis of the FB method in nonconvex settings, to \cite{bodard2025second} for its use to show that second-order methods converge to second-order stationary points, to \cite{MR3654183} for a use of the envelope in DC programming, and to \cite{MR4165947} where it is combined with inexact proximal steps.

In the following, we state our blanket assumptions and properties of the DYE \eqref{DY-env-x} relevant in our analysis.

\begin{assumption} \label{a:ass1}
    Consider the optimisation problem in \eqref{DY-problem}, and suppose that\begin{enumerate}
        \item[(i)] $f: \R^d \to \R$ is $L_f$-smooth and $\rho_f$-weakly convex, 
        and that $h: \R^d \to \R$ is $L_h$-smooth,
        \item[(ii)] $g: \R^d \to \R \cup \{+\infty\}$ is proper lower semicontinuous and $\rho_g$-weakly convex,
        \item[(iii)] the problem in \eqref{DY-problem} has a nonempty set of minimisers. In particular, $\varphi$ is bounded below.
    \end{enumerate}
\end{assumption}

In Section~\ref{s:stepsize-experiment} we show a numerical example that motivates the study of the DY method even when both $f$ and $h$ are smooth, a setting in which the FB method can also be employed.

\begin{lemma} \label{l:envelope-prop}  Let Assumption~\ref{a:ass1} hold, and let $\gamma \in (0, \min\{\rho_g^{-1},(L_f+L_h)^{-1}\})$. Then, 
    \begin{enumerate}
        \item[(i)]  $\DYE$  is locally Lipschitz continuous.
        \item[(ii)]  $\inf \varphi = \inf \DYE$.
        \item[(iii)] $\varphi$ is level-bounded if and only if $\DYE$ is level-bounded.
    \end{enumerate}
\end{lemma}

\begin{proof}
   Given that $\varphi_\gamma^{\rm DY} \circ \prox_{\gamma f} = V_\gamma $, then item (i) can be deduced similarly to \cite[Proposition 3.2]{themelis2020douglas}, item (ii) can be obtained from \cite[Theorem 4.1 1.]{liu2019envelope} (see also \cite[Theorem 4.4]{MR3845278}), and item (iii) is analogous to \cite[Theorem 3.4]{themelis2020douglas}.
\end{proof}

\begin{remark}
In \cite{bian2021three}, a different merit function is introduced to analyse the DY method. In the context of the iterative scheme \eqref{DY}, that merit function, denoted $\mathfrak{D}_{\gamma}$, satisfies $\mathfrak{D}_{\gamma}(x^k,y^k,z^k) = V_\gamma(z^k) - \frac{1}{\gamma}\|z^{k+1} - z^k\|^2$ for all $k\geq 1$. Hence, $\mathfrak{D}_{\gamma}(x^k,y^k,z^k) \leq V_\gamma(z^k) \leq \varphi(x^k)$, where the second inequality follows from taking $z = z^k$ and $y = x^k$ in \eqref{DY-z}. In this manner, the merit function  $V_\gamma$ (and thus, the DYE in \eqref{DY-env-x}) is a tighter approximation of $\varphi$ than $\mathfrak{D}_{\gamma}$. As a consequence, the range of stepsizes for which the method in \eqref{DY} is convergent is larger when $V_\gamma$  is used in the analysis instead of $\mathfrak{D}_{\gamma}$ (see, e.g., \cite[Figure 1]{alcantara2024four}). This fact has a direct consequence in the numerical performance of the method. Furthermore, we highlight that Lipschitz continuity of the gradients is key to establishing Lemma~\ref{l:envelope-prop}(ii) above, even for the forward-backward envelope, that is, when $f=0$. Indeed, if $\nabla h$ is not Lipschitz continuous, Lemma~\ref{l:envelope-prop}(ii)  may fail as  \cite[Remark 5.4]{aragon2025coderivative} shows.

\end{remark}


We denote the \emph{forward-backward map}, the solution map of the problem in \eqref{FB-env}, by \begin{equation} \label{FB-env-sol}
   \R^d \times \R_{++} \ni (x,\gamma) \mapsto\RFB(x, \gamma) := \prox_{\gamma g}\big(x- \gamma \nabla f(x)\big),
\end{equation} and the \emph{DY map} as \begin{equation} \label{eq:DY-solution-map}
   \RDY := Y^{\rm FB}_{f+h,g}.
\end{equation} In the following, we list two properties of the DY map we will use in our analysis. First, we state a continuity property of $\RDY$ (cf. \cite[Theorem 4.7(i)]{MR3845278}).

\begin{lemma}\label{DY-continuity-prop}
    Let Assumption~\ref{a:ass1} hold. Then, for  $\gamma \in (0, \rho_g^{-1})$, $\RDY(\cdot, \gamma)$ is Lipschitz continuous with constant $\tfrac{1 + \gamma (L_f+L_h)}{1 - \gamma\rho_g}$. 
\end{lemma}

\begin{proof}
It suffices to show the property holds for $h=0$. From the first-order optimality condition of \eqref{FB-env-sol}, for $i=1,2$, we have: \begin{equation*}\gamma^{-1}(u_i - \gamma \nabla f(u_i)) + (\rho_g - \gamma^{-1})\RFB(u_i, \gamma) \in \Big( \partial g (\cdot) + \rho_g I\Big)\Big(\RFB(u_i, \gamma)\Big).\end{equation*} Since $\partial g (\cdot) + \rho_g I$ is  monotone, \begin{equation*}\begin{array}{rcl}0 & \le & \langle \RFB(u_1, \gamma)  - \RFB(u_2, \gamma), \gamma^{-1}\big(u_1 -  \gamma \nabla f(u_1) \big) - \gamma^{-1}\big( u_2  - \gamma\nabla f(u_2) \big) \rangle \\ & & \quad + (\rho_g - \gamma^{-1})\langle \RFB(u_1, \gamma)  - \RFB(u_2, \gamma), \RFB(u_1, \gamma)  - \RFB(u_2, \gamma) \rangle,\end{array}\end{equation*} which implies \begin{align*}
     (\gamma^{-1} - \rho_g) \|\RFB(u_1, \gamma)  - \RFB(u_2, \gamma)\|^2  \leq & \gamma^{-1}\langle \RFB(u_1, \gamma)  - \RFB(u_2, \gamma), u_1 - u_2 \rangle \\ & + \: \langle \RFB(u_1, \gamma)  - \RFB(u_2, \gamma), \nabla f(u_2) - \nabla f(u_1) \rangle  \\ \leq  & (\gamma^{-1} + L_f)\|\RFB(u_1, \gamma)  - \RFB(u_2, \gamma)\| \| u_1 - u_2\|,
\end{align*} where the last line follows from Lipschitz continuity of $\nabla f$. Since $\gamma^{-1} > \rho_g$, then the desired estimate follows.\end{proof}


The following result provides an explicit formula for the gradient of the DYE, and follows directly from  \eqref{DY-FB-identity-2} and \cite[Lemma 3.3]{atenas2025understanding} (see also \cite[Theorem 4.7(ii)]{MR3845278}).

\begin{lemma}[Gradient of DYE] \label{l:DYE-subdiff} Let Assumption~\ref{a:ass1} hold. Suppose in addition that $f$ and $h$ are twice continuously differentiable with Lipschitz continuous Hessians. Then, for any $\gamma \in (0, \min\{\rho_g^{-1},(L_f+L_h)^{-1}\})$ and $x \in \R^d$, \begin{equation*}
		\nabla \DYE(x) = \gamma^{-1}\big( I - \gamma \nabla^2 (f+h)(x) \big)\big(x - \RDY(x,\gamma)\big).
	\end{equation*}   
\end{lemma}

Although the variational properties established above follow from those of their forward-backward counterparts, Section~\ref{s:stepsize-experiment} demonstrates that the empirical performance of the DY and the FB methods is not the same, and the former is more robust with respect to parameter selection than the latter.

\section{Convergence of the Davis--Yin Splitting Method for Nonconvex Optimisation} \label{s:damped-DY}

In this section, we propose a new perspective for analysing proximal splitting methods for nonconvex optimisation, one that only relies on the shadow sequence. We use it to study the convergence properties of a damped generalisation of the DY method, the variant of the algorithm we need for the analysis of saddle point avoidance in Section~\ref{s:saddle-avoidance} (see Remark~\ref{r:needed}). 

\subsection{Davis--Yin Splitting with Damped Shadow Sequence: Formulation}

We start by defining the damped DY method we analyse. Suppose that Assumption~\ref{a:ass1} holds. Let $x^0\in\R^d$, let $\gamma \in (0, \min\{\rho_g^{-1},\rho_f^{-1}\})$, and set
$z^0 := x^0 + \gamma \nabla f (x^0)$. For $k\geq 0$, consider the iterates of the DY method with damping in the shadow sequence, defined as:
\begin{equation} \label{DY-damped-z} \tag{D-DY}
    \left\{\begin{array}{rcl}
        y^k     & = & \prox_{\gamma g}\big(2x^k-z^k-\gamma\nabla h(x^k)\big) \\[2pt]
        x^{k+1} & = & (1-\alpha)\,x^k
                      + \alpha\,\prox_{\gamma f}\big(z^k+\lambda(y^k-x^k)\big) \\[2pt]
        z^{k+1} & = & x^{k+1} + \gamma\nabla f(x^{k+1}).
    \end{array}\right.
\end{equation} Observe that, as $f$ is smooth, \begin{equation} \label{eq:invariant-z}
    z^k = x^k + \gamma\nabla f(x^k) \iff
    x^k = \prox_{\gamma f}(z^k), \qquad k \geq 0.
\end{equation} Therefore, if $\alpha=1$, then $z^{k+1}=z^k+\lambda(y^k-x^k)$ for all $k\geq0$, and
\eqref{DY-damped-z} becomes \eqref{DY}. In the next result, we show an equivalent reformulation of \eqref{DY-damped-z} independent of the sequence $(z^k)_{k\geq 1}$.

\begin{proposition}[Damped Davis--Yin splitting: alternative form]
\label{p:damped-DY-z}
Let Assumption~\ref{a:ass1} hold, and let $\gamma\in(0,\min\{\rho_g^{-1},\rho_f^{-1}\})$, $\alpha\in(0,1]$ and
$\lambda>0$. Let $x^0\in\R^d$, set $z^0:=x^0+\gamma\nabla f(x^0)$, and 
consider the iterates $(x^k)_{k\geq1}$, $(y^k)_{k\geq1}$, and $(z^k)_{k\geq1}$ generated by \eqref{DY-damped-z}. Then the sequences $(x^k)_{k\geq1}$ and $(y^k)_{k\geq1}$ can be equivalently generated via
\begin{equation} \label{eq:damped-DY-shadow}
        \left\{\begin{array}{rcl}
            y^k &=& \prox_{\gamma g}\big(x^k - \gamma \nabla (f+h)(x^k)\big),\\
            x^{k+1} &= &(1-\alpha)x^k + \alpha \prox_{\gamma f}\big((1-\lambda)x^k
                        + \gamma \nabla f(x^k) + \lambda y^k\big).
        \end{array}\right.
\end{equation}
\end{proposition}

\begin{proof}
Fix $k\geq0$. Since $f$ is $\rho_f$-weakly convex, and $\gamma<\rho_f^{-1}$, the third line of \eqref{DY-damped-z}, together with the choice
of $z^0$, gives \eqref{eq:invariant-z}. Then $\gamma\nabla f(x^k)=z^k-x^k$, so that $2x^k - z^k - \gamma\nabla h(x^k)
    = x^k - \gamma\nabla (f+h)(x^k)$, and the first line of \eqref{DY-damped-z} is equivalent to the first line in \eqref{eq:damped-DY-shadow}.  Using \eqref{eq:invariant-z} once more, $z^k + \lambda\big(y^k-x^k\big)
    = (1-\lambda)x^k + \gamma\nabla f(x^k) + \lambda y^k$, so that the second line of
\eqref{DY-damped-z} is equivalent to the second line in \eqref{eq:damped-DY-shadow}. As both
schemes start at $x^0$, the generated sequences coincide.
\end{proof}

\begin{remark}[On the computational cost of damping] 
An iteration of \eqref{DY} evaluates $\prox_{\gamma f}$, $\prox_{\gamma g}$ and
$\nabla h$ once each. With damping, in principle, we need an extra gradient evaluation of $f$ at $x^{k+1}$ in the third line of \eqref{DY-damped-z}. However, in common applications of interest, for example, the Tikhonov/ridge regularisation (a multiple of $\|\cdot\|_2^2$), we can recover the same per-iteration cost of \eqref{DY}. 
We formalise this claim in Proposition~\ref{p:damped-DY-quadratic} below.

\end{remark}

\begin{proposition}[Damped Davis--Yin splitting with quadratic $f$] \label{p:damped-DY-quadratic}
Let Assumption~\ref{a:ass1} hold, and let $\gamma\in(0,\min\{\rho_g^{-1},\rho_f^{-1}\})$, $\alpha\in(0,1]$ and $\lambda>0$.
Suppose in addition that $f$ is a quadratic function, that is, for some symmetric
$F \in \R^{d\times d}$, $b \in \R^d$ and $c \in \R$, \begin{equation*}
    f(x)=\tfrac12\langle Fx,x\rangle + \langle b,x\rangle + c, \qquad x \in \R^d.
\end{equation*} Let $x^0\in\R^d$, set
$z^0 := x^0 + \gamma \nabla f(x^0)$, and consider the sequences $(x^k)_{k\geq1}$, $(y^k)_{k\geq1}$, and $(z^k)_{k\geq1}$ generated by \eqref{DY-damped-z}. Then, these sequences are equivalently generated by \eqref{DY} with relaxation parameter $\alpha \lambda$, that is, for $k\geq0$, they conform to \begin{equation} \label{DY-damped-z-quadratic}
    \left\{\begin{array}{rcl}
        x^k     & = & \prox_{\gamma f}(z^k) \\
        y^k     & = & \prox_{\gamma g}\big(2x^k-z^k-\gamma\nabla h(x^k)\big) \\
        z^{k+1} & = & z^k + \alpha\lambda\,(y^k-x^k).
    \end{array}\right.
\end{equation}
\end{proposition}
\begin{proof} Let $(x^k)_{k\geq1}$, $(y^k)_{k\geq1}$, and $(z^k)_{k\geq1}$ be generated by \eqref{DY-damped-z}. In view of \eqref{eq:invariant-z}, it suffices to show that the third line in \eqref{DY-damped-z-quadratic} holds true. Since $f$ is quadratic and $\rho_f$-weakly convex with $\gamma<\rho_f^{-1}$, the matrix
$I+\gamma F$ is positive definite, and $\prox_{\gamma f}$ is the affine map
$\prox_{\gamma f}(z)=(I+\gamma F)^{-1}(z-\gamma b)$. From the second line in \eqref{DY-damped-z} and \eqref{eq:invariant-z},  
\begin{align*}
    x^{k+1}
    &= (1-\alpha)\prox_{\gamma f}(z^{k}) + \alpha \prox_{\gamma f}(z^k+\lambda(y^k-x^k))\\
    &= \prox_{\gamma f}(z^k + \alpha\lambda\,(y^k-x^k)).
\end{align*} Invoking \eqref{eq:invariant-z} once again, $\prox_{\gamma f}$ and $I + \gamma \nabla f$ are inverse to each
other, then the third line in \eqref{DY-damped-z}  yields $z^{k+1} = (I + \gamma \nabla f)(x^{k+1}) =z^k + \alpha\lambda\,(y^k-x^k)$, concluding the proof.
\end{proof}

The iterative scheme \eqref{DY-damped-z} and Proposition~\ref{p:damped-DY-z} motivate the definition of the following maps. Consider a stepsize $\gamma > 0$ and a damping parameter $\alpha \in (0,1)$. Define the damped shadow DY iteration operator as \begin{equation} \label{eq:damped-DY-operator}
    \R^d \ni x \mapsto {\rm DY}_{\gamma,\alpha}(x) := (1-\alpha)x + \alpha T_{\gamma}(x),
\end{equation} where $T_\gamma$ is defined as \begin{equation} \label{DY:T-operator}
    T_{\gamma}(x)  :=  \prox_{\gamma f}\big( (1-\lambda)x +  \gamma \nabla f(x) + \lambda Y_{\gamma}(x) \big),
\end{equation}  and $Y_{\gamma}(x) = \RDY(x,\gamma) = \prox_{\gamma g} (x - \gamma \nabla(f+h)(x)) $. In view of Proposition~\ref{p:damped-DY-z},  the damped shadow DY sequence  $(x^k)_{k\geq 1}$ generated via \eqref{DY-damped-z} is characterised by  \begin{equation} \label{eq:damped-shadow-DY}
        x^{k+1} = {\rm DY}_{\gamma,\alpha}(x^k), \qquad k\geq 1.
    \end{equation} Observe that under Assumption~\ref{a:ass1}, for $\gamma \in (0, \min\{\rho_g^{-1}, \rho_f^{-1}\})$, the first-order optimality condition of the optimisation problem in \eqref{DY:T-operator} gives \begin{equation} \label{T-OC}
    \gamma^{-1}\big( (1-\lambda)x + \gamma \nabla f(x) + \lambda Y_{\gamma}(x) - T_{\gamma}(x) \big) = \nabla f\big(T_{\gamma}(x) \big).
\end{equation}The reformulation in Proposition~\ref{p:damped-DY-z} is not only convenient for the convergence analysis, but also to characterise stationary points of $\varphi$ in terms of fixed points of the shadow iteration operator.

\begin{proposition}[Davis--Yin splitting: stationary points are fixed points] \label{DY:critical-characterization-I} Let Assumption~\ref{a:ass1} hold.  Then, for any $\gamma \in (0,  \min\{\rho_g^{-1}, \rho_f^{-1}\})$, the following statements are equivalent. 
    
    \begin{enumerate}
        \item[(i)] $\bar{x}$ is a fixed-point of $T_{\gamma}$. 
        \item[(ii)] $\bar{x}$ is a fixed-point of $Y_{\gamma}$. 
        \item[(iii)] $\bar{x}$ is a stationary point of $\varphi = f+g+h$.
    \end{enumerate}  
   
\end{proposition}

\begin{proof}
    Let $x \in \R^d$, and define $x^+ = T_{\gamma}(x)$. From \eqref{T-OC}, $x^+$ is the unique point satisfying  \begin{equation*}
        Y_{\gamma}(x) = \displaystyle\frac{1}{\lambda}\left(x^+ + \gamma \big( \nabla f(x^+) -\nabla f(x) \big) - (1-\lambda)x\right).
    \end{equation*} Furthermore, the first-order optimality condition of $Y_\gamma(x)$ reads
    \begin{equation*}
        \dfrac{1}{\gamma}\big(x-\gamma \nabla (f+h)(x) -Y_{\gamma}(x) \big) \in \partial g \big(Y_{\gamma}(x)\big).
    \end{equation*} If $T_{\gamma}(\bar x) = \bar x$, then \begin{align*}
        & Y_{\gamma}(\bar x) = \frac{1}{\lambda}\left( \bar{x} + \gamma \big( \nabla f(\bar{x}) -\nabla f(\bar{x}) \big) - (1-\lambda)\bar{x}\right) = \bar{x}, \text{~and~}\\
       & - \nabla (f+h)(\bar x) = \dfrac{1}{\gamma}\big(\bar x-\gamma \nabla (f+h)(\bar x) -Y_{\gamma}(\bar x) \big) \in \partial g( \bar x ) ,
    \end{align*}that is, $\bar x$ is a stationary point of $\varphi$. The converse follows similarly, since $T_{\gamma}(\bar x)$ and $Y_{\gamma}(\bar x)$  satisfy their corresponding first-order optimality conditions. Indeed, if $\bar{x}$ is a stationary point of $\varphi$, then  \begin{equation*}
        \dfrac{1}{\gamma}\big(\bar{x}-\gamma \nabla (f+h)(\bar{x}) -\bar{x} \big) \in \partial g (\bar{x}),
    \end{equation*} which  is equivalent to $\bar x = Y_{\gamma}(\bar x)$ as $\gamma<\rho_g^{-1}$. In turn, this identity also implies \begin{equation*}
        \dfrac{1}{\gamma}\big( (1-\lambda)\bar x + \gamma \nabla f(\bar x) + \lambda \bar x - \bar x \big) = \nabla f(\bar x),
    \end{equation*} which, combined with \eqref{T-OC},  implies $\bar x = T_{\gamma}(\bar x)$.
\end{proof}

\subsection{Davis--Yin Splitting with Damped Shadow Sequence: Convergence}

 In order to show convergence of the sequence generated by \eqref{DY-damped-z}, we revisit a strategy introduced by Opial in \cite{MR211301} to analyse fixed-point iterations, namely, \emph{asymptotic regularity}. We say an operator $S: \R^d \to \R^d$ is asymptotically regular if for any  $u^0\in\R^d$,  the sequence $(u^k)_{k\geq1}$ generated by $u^{k+1} = S(u^k)$ for all $k\geq1$ satisfies $u^k - S(u^k) \to 0$. Opial's theorem \cite[Theorem 1]{MR211301} holds true for nonexpansive operators, suitable for convex optimisation. In the next result, we adapt this idea when the operator is not necessarily nonexpansive, which is the case in nonconvex optimisation.

    \begin{lemma} \label{l:Opial}
    Let Assumption~\ref{a:ass1} hold, and let $\gamma \in (0, \min\{\rho_g^{-1}, \rho_f^{-1}\})$ be a stepsize,  $\alpha \in (0,1)$ a damping parameter, and $\lambda >0$ a relaxation parameter. Suppose that the damped shadow DY sequence $(x^k)_{k\geq1}$ in \eqref{DY-damped-z} has a nonempty set of cluster points, and that  $x^k - {\rm DY}_{\gamma,\alpha}(x^k) \to 0$. Then any cluster point of $(x^k)_{k\geq1}$ is a stationary point of $\varphi$.
    \end{lemma}

    \begin{proof}
    Asymptotic regularity implies that $T_\gamma(x^k) - x^k \to 0$. Since $\nabla f$ is Lipschitz continuous, then $\nabla f(T_\gamma(x^k)) - \nabla f(x^k) \to 0$.  Let $x^\star$ be a cluster point of $(x^k)_{k\geq1}$, such that, up to a subsequence, $x^k \to x^\star$. Then, 
    it follows from \eqref{T-OC} that \begin{equation*}
        \frac{\gamma}{\lambda} \big(\nabla f(T_\gamma(x^k))-\nabla f(x^k)\big) + \frac{1}{\lambda}\big(T_\gamma(x^k) - x^k \big) +  x^k =  Y_\gamma(x^k).
    \end{equation*} As $x\mapsto \prox_{\gamma g}(x)$ is continuous, continuity of $\nabla (f+h)$ implies that  $x\mapsto Y_\gamma(x)$ is continuous as well. Then, taking the limit in the identity above yields $x^\star = Y_\gamma(x^\star)$, 
    and thus $x^\star$ is a stationary point of $\varphi$ from Proposition~\ref{DY:critical-characterization-I}.
    \end{proof}  
    
    In this manner, the convergence analysis of \eqref{DY-damped-z}  relies on verifying that the conditions in Lemma~\ref{l:Opial} are satisfied. The core of the argument is based on a sufficient descent condition for the DYE \eqref{DY-env-x}. Different from \cite{themelis2020douglas, bian2021three, alcantara2024four, alcantara2025relaxed}, this sufficient condition needs to hold alongside the damped shadow DY sequence, and not for the traditional fixed-point sequence $(z^k)_{k\geq1}$ from the iteration scheme \eqref{DY}.  Its proof is inspired by \cite[Theorem 3.5, Theorem 3.9]{alcantara2024four}, which in turn was motivated by \cite{themelis2020douglas}. For convenience, we define \begin{equation} \label{eq:Q}
        Q_{\gamma(f+h)}(y;x) := (f+h)(x) + \left<\nabla (f+h)(x), y-x\right> + \frac{1}{2\gamma}\|y-x\|^2,
    \end{equation} in such a way that \begin{equation} \label{eq:DYE-Q}
        \DYE(x) = \inf_{y \in \R^d}\{Q_{\gamma(f+h)}(y;x) + g(y)\}.
    \end{equation} From the definition of the map $Y_\gamma$, $y = Y_\gamma(x)$ solves the optimisation problem in \eqref{eq:DYE-Q}.
    
    We start the proof of the sufficient descent condition with two  technical results.

    \begin{lemma} \label{l:suff-descent}

Suppose Assumption~\ref{a:ass1} holds. Let $(x^k)_{k\geq1}$ be the damped shadow DY sequence of \eqref{DY-damped-z}. Then for any $\gamma \in (0, \min\{\rho_g^{-1},\rho_f^{-1}\})$,  $ L \geq L_f$, $L>\rho_f$, and all $k\geq1$, \begin{equation*}
        \DYE(x^{k+1}) \leq Q_{\gamma (f+h)}(Y_\gamma(x^k);x^{k+1}) + g(Y_\gamma(x^k)),
    \end{equation*} and \begin{align*}
            \DYE(x^k)  \geq & Q_{\gamma(f+h)}(Y_\gamma(x^k);x^{k+1})  + g(Y_\gamma(x^k))   + q_1^k + q_2^k \\ & + \frac{1}{2(L - \rho_f)}\|\nabla f(x^k) - \nabla f(x^{k+1})\|^2  - \left( \frac{L_h}{2} + \frac{1}{2\gamma} + \frac{\rho_f L}{2(L-\rho_f)}\right)\|x^k - x^{k+1}\|^2,
        \end{align*} where
        \begin{equation}\label{eq:q_1^k}
            q_1^k := \left< \nabla f(x^{k+1}) - \nabla f(x^k) - \frac{1}{\gamma}(x^{k+1}-x^k), x^k - Y_\gamma(x^k) \right>
        \end{equation} and \begin{equation}\label{eq:q_2^k}
            q_2^k := \left< \nabla h(x^k) - \nabla h(x^{k+1}), Y_\gamma(x^k) - x^{k+1} \right>.
        \end{equation}
    \end{lemma}

    \begin{proof}
        The first estimate follows directly from \eqref{eq:Q}-\eqref{eq:DYE-Q} by taking $y = Y_\gamma(x^k)$. The second estimate is \cite[Lemma 3.4]{alcantara2024four}  with  $x = x^k$, $y^+ = Y_\gamma(x^k)$, $\hat{x} = x^{k+1}$,  $p = 0$, $\rho_p = 0$ and $\alpha \gets \gamma$ (therein, $\alpha$ is the stepsize).
    \end{proof}

    \begin{lemma} \label{l:estimes-q}
    In the context of Lemma~\ref{l:suff-descent}, the following hold for all $k\geq1$: \begin{align*}
        q_1^k \geq & \left[-\frac{2L_f}{\lambda\alpha}  - \frac{L_f^2\gamma}{\lambda}\left(\frac{1-\alpha}{\alpha}\right) + \frac{1}{\lambda\alpha \gamma} \right]\|x^k - x^{k+1}\|^2  - \frac{\gamma}{\lambda}\|\nabla f(x^k) - \nabla f (x^{k+1})\|^2, \text{~and~} \\
         q_2^k \geq & 
        -\frac{L_h}{\lambda\alpha}(|1-\lambda\alpha|+L_f\gamma)\|x^k - x^{k+1}\|^2.
    \end{align*}         
    \end{lemma}

    \begin{proof}
From \eqref{eq:damped-DY-operator} and \eqref{eq:damped-shadow-DY}, \begin{equation} \label{eq:x-T}
    x^k - T_\gamma(x^k) = \frac{1}{\alpha}(x^k - x^{k+1}),
\end{equation} which combined with \eqref{T-OC} yields \begin{equation} \label{eq:x-Y}
    \begin{aligned}
    x^k - Y_\gamma(x^k) & = \frac{1}{\lambda}\left[x^k - T_\gamma(x^k) + \gamma \big( \nabla f(x^k) - \nabla f(T_\gamma(x^k))\big)\right]\\ 
    & = \frac{1}{\lambda\alpha}(x^k-x^{k+1}) + \frac{\gamma}{\lambda}\big(\nabla f(x^k) - \nabla f(T_\gamma(x^k))\big).
    \end{aligned}
\end{equation} By substituting \eqref{eq:x-Y} in \eqref{eq:q_1^k}, we get \begin{align*}
    q_1^k  = & \frac{1}{\lambda\alpha}\left< \nabla f(x^{k+1}) - \nabla f(x^k),x^k - x^{k+1} \right> \\ &  + \frac{\gamma}{\lambda}\left<\nabla f(x^{k+1}) - \nabla f(x^k), \nabla f(x^k) - \nabla f(T_\gamma(x^k)) \right>\\
    & + \frac{1}{\lambda\alpha\gamma}\|x^k-x^{k+1}\|^2- \frac{1}{\lambda}\left< x^{k+1}-x^k, \nabla f(x^k) - \nabla f(T_\gamma(x^k)) \right>\\ %
    \geq & -\frac{1}{\lambda\alpha}\|\nabla f(x^{k+1}) - \nabla f(x^k)\|\|x^k-x^{k+1}\| - \frac{\gamma}{\lambda}\| \nabla f(x^{k+1})  - \nabla f(x^k)\|^2  \\  & - \frac{\gamma}{\lambda} \|\nabla f(x^{k+1}) - \nabla f(x^k)\|\| \nabla f(x^{k+1}) - \nabla f(T_\gamma(x^k))\|
    \\ & + \frac{1}{\lambda\alpha\gamma}\|x^k-x^{k+1}\|^2- \frac{1}{\lambda}\| x^{k+1}-x^k\| \| \nabla f(x^k) - \nabla f(T_\gamma(x^k))\| \\ %
    \geq & -\frac{L_f}{\lambda\alpha}\|x^k-x^{k+1}\|^2 - \frac{\gamma}{\lambda}\| \nabla f(x^{k+1})  - \nabla f(x^k)\|^2   - \frac{\gamma L_f^2}{\lambda}\|x^{k+1} - x^k \| \|  x^{k+1} - T_\gamma(x^k) \| \\ & + \frac{1}{\lambda\alpha\gamma}\|x^k-x^{k+1}\|^2 - \frac{L_f}{\lambda}\| x^{k+1}-x^k\| \| x^k - T_\gamma(x^k)\|, 
    \end{align*} where the first inequality follows from adding and subtracting $\nabla f(x^{k+1})$ in the inner product in the second line and applying the Cauchy--Schwarz inequality, and the second inequality follows from using Lipschitz continuity of $\nabla f$. Moreover, from \eqref{eq:x-T}, $T_\gamma(x^k) - x^{k+1} = (1-\alpha)(T_\gamma(x^k)-x^k) = \left(\frac{1-\alpha}{\alpha}\right)(x^{k+1}-x^k)$, then \begin{align*}
        q_1^k \geq & -\frac{L_f}{\lambda\alpha}\|x^k-x^{k+1}\|^2 - \frac{\gamma}{\lambda}\| \nabla f(x^{k+1})  - \nabla f(x^k)\|^2   - \frac{\gamma L_f^2}{\lambda}\left(\frac{1-\alpha}{\alpha}\right)\|x^{k+1} - x^k \|^2 \\  & + \frac{1}{\lambda\alpha\gamma}\|x^k-x^{k+1}\|^2 - \frac{L_f}{\lambda\alpha}\| x^{k+1}-x^k\|^2,
    \end{align*}  from where the first estimate follows. Next, in view of \eqref{eq:x-Y}, \begin{equation*}
        Y_\gamma(x^k) - x^{k+1} = \left(\frac{1}{\lambda\alpha}-1 \right)(x^{k+1}-x^k) + \frac{\gamma}{\lambda}\big(\nabla f(T_\gamma(x^k))-\nabla f(x^k)\big),
    \end{equation*} then substituting this identity in \eqref{eq:q_2^k} yields \begin{align*}
        q_2^k = & \left(\frac{1}{\lambda\alpha}-1 \right)\left< \nabla h(x^k) - \nabla h(x^{k+1}), x^{k+1}-x^k \right> \\  & 
       + \frac{\gamma}{\lambda}\left< \nabla h(x^k) - \nabla h(x^{k+1}), \nabla f(T_\gamma(x^k))-\nabla f(x^k) \right> \\%
       \geq & -\left|\frac{1}{\lambda\alpha}-1 \right|L_h \|x^k-x^{k+1}\|^2  
       - \frac{\gamma}{\lambda}L_hL_f\| x^k - x^{k+1} \| \|T_\gamma(x^k)-x^k\|\\%
       = & -\left|\frac{1}{\lambda\alpha}-1 \right|L_h \|x^k-x^{k+1}\|^2  
       - \frac{\gamma}{\lambda\alpha}L_hL_f\| x^k - x^{k+1} \|^2,
    \end{align*} where the inequality is obtained by using Lipschitz continuity of $\nabla h$ and $\nabla f$, and the fact that $\beta \geq -|\beta|$ for all $\beta \in \R$, and the equality follows from \eqref{eq:x-T}. Thus we obtain the estimate for $q_2^k$. \end{proof}

In the following result, we show that the DYE satisfies a sufficient descent condition alongside the damped shadow DY sequence in \eqref{DY-damped-z}. 

 \begin{theorem}[Damped DY splitting: shadow sufficient descent] \label{t:DYE-suff-descent} Let Assumption~\ref{a:ass1} hold. Let  $\alpha \in (0,1)$ be a damping parameter, and $\lambda >0$ a relaxation parameter such that $\lambda\alpha < 2$. Consider the damped shadow DY sequence $(x^k)_{k \geq 1}$ generated via \eqref{DY-damped-z}. Then, for all sufficiently small $\gamma >0$, there exists $c > 0$, such that for all  $k \geq 1$,
\begin{equation} \label{eq:DYE-suff-descent}
    \varphi_\gamma^{\rm DY}(x^{k+1}) + \frac{c}{\gamma}\|x^{k+1}-x^k\|^2 \leq \varphi_\gamma^{\rm DY}(x^k).
\end{equation}

 \end{theorem}
 
  \begin{proof}
 	Let $L \geq L_f$ with $L > \rho_f$. By subtracting the two estimates in Lemma~\ref{l:suff-descent}, and using Lemma~\ref{l:estimes-q}, we get \begin{equation}
 		\label{eq:suff-descent-aux-1}
 		\DYE(x^k) - \DYE(x^{k+1}) \geq  \left[ - \frac{\gamma}{\lambda} + \frac{1}{2(L - \rho_f)}\right]\|\nabla f(x^k) - \nabla f(x^{k+1})\|^2  + C(\gamma,L)\|x^k - x^{k+1}\|^2,
 	\end{equation}where \begin{multline*}
 		C(\gamma,L) := -\left(\tfrac{L_f^2}{\lambda}\left(\tfrac{1-\alpha}{\alpha}\right) + \tfrac{L_fL_h}{\lambda\alpha} \right)\gamma - \left( \tfrac{2L_f}{\lambda\alpha}  + \left(\tfrac{1}{2} + \tfrac{|1-\lambda\alpha|}{\lambda\alpha} \right) L_h + \tfrac{\rho_f L}{2(L-\rho_f)}\right)  + \left( \tfrac{1}{\lambda\alpha} - \tfrac{1}{2}\right) \tfrac{1}{\gamma}.
 	\end{multline*} 
 	Observe that the constant multiplying $\|\nabla f(x^k) - \nabla f(x^{k+1})\|^2$ in \eqref{eq:suff-descent-aux-1} changes sign at $\gamma = \hat{\gamma}(L) := \frac{\lambda}{2(L-\rho_f)}$. Let $A = \frac{L_f(L_f+L_h)}{\lambda\alpha}$, $B = \frac{2L_f}{\lambda\alpha} - \frac{L_f}{2} + \Big(\frac12+\frac{|1-\lambda\alpha|}{\lambda\alpha}\Big)L_h>0$, $\omega = \tfrac{1}{\lambda\alpha} - \tfrac{1}{2}$ and $\Delta = \frac{\lambda\big(B+\sqrt{B^2+4\omega A}\big)}{4\omega}$. We separate the analysis into two cases. First, suppose that $L_f - \rho_f > \Delta$ and set $\hat{\gamma}^+ = \frac{2\omega}{B + \sqrt{B^2+4A\omega}}$. Then $\hat{\gamma}(L_f) < \hat{\gamma}^+$. In this case, we may take $L = L_f$, so that $C(\gamma,L_f)$ is well-defined. Set  \[\hat{C}_1(\gamma) := \gamma C(\gamma,L_f) + \gamma\left[- \frac{\gamma}{\lambda} + \frac{1}{2(L_f - \rho_f)}\right]L_f^2.\] A direct computation shows that $\hat{C}_1(\gamma) = -A \gamma^2 - B\gamma + \omega$. As $A \geq 0$ and $B > 0$, $\hat{C}_1$ is decreasing on $[0,+\infty)$, its unique positive root is $\hat{\gamma}^+$, and $\hat{C}_1(0) = \omega > 0$ from the assumption $\lambda\alpha < 2$. Hence, $\hat{C}_1(\gamma) > 0$ for all $\gamma \in [0,\hat{\gamma}^+)$. Let now $\gamma \in (0,\hat{\gamma}^+)$. If $\gamma \geq \hat{\gamma}(L_f)$, then Lipschitz continuity of $\nabla f$ applied to \eqref{eq:suff-descent-aux-1} yields \begin{equation*}
 			\DYE(x^k) - \DYE(x^{k+1}) \geq \frac{\hat{C}_1(\gamma)}{\gamma}\|x^k - x^{k+1}\|^2.
 		\end{equation*} If $\gamma < \hat{\gamma}(L_f)$, we can bound below the term associated with $\|\nabla f(x^k) - \nabla f(x^{k+1})\|^2$ in \eqref{eq:suff-descent-aux-1} by zero, so that \begin{equation*}
 			\DYE(x^k) - \DYE(x^{k+1}) \geq \frac{\gamma C(\gamma,L_f)}{\gamma}\|x^k - x^{k+1}\|^2.
 		\end{equation*} The map $\gamma \mapsto \gamma C(\gamma,L_f)$ is a concave quadratic function whose value at $0$ is $\omega > 0$, and whose value at $\hat{\gamma}(L_f)$ is $\hat{C}_1(\hat{\gamma}(L_f)) > 0$. Thus, it is positive on $[0,\hat{\gamma}(L_f)]$. Hence, \eqref{eq:DYE-suff-descent} holds in this case for $\gamma \in (0, \min\{ \rho_g^{-1}, \rho_f^{-1}, \hat{\gamma}^+\} )$ and $c = \min\{\gamma C(\gamma,L_f), \hat{C}_1(\gamma)\}$. Second, suppose that $L_f - \rho_f \leq \Delta$, and consider $\gamma \in (0,\hat{\gamma}(L))$. Then we can bound below the term associated with $\|\nabla f(x^k) - \nabla f(x^{k+1})\|^2$ in \eqref{eq:suff-descent-aux-1} by zero. In this case, we define \begin{multline}
 		\label{eq:C_2-hat}
 		\hat{C}_2(L) :=  - \left(\tfrac{L_f^2}{\lambda}\left(\tfrac{1-\alpha}{\alpha}\right) + \tfrac{L_fL_h}{\lambda\alpha} \right)\hat{\gamma}^2-\left(\tfrac{2L_f}{\lambda\alpha} + \left(\tfrac{1}{2} + \tfrac{|1-\lambda\alpha|}{\lambda\alpha} \right) L_h + \tfrac{\rho_f L}{2(L-\rho_f)}\right) \hat{\gamma}  +  \tfrac{1}{\lambda\alpha} - \tfrac{1}{2}
 	\end{multline} such that $\hat{C}_2(L) = \hat{\gamma}(L)C(\hat{\gamma}(L),L)$. Then  $\hat{C}_2(L) = \frac{1}{[2(L-\rho_f)]^2}p(L-\rho_f)$, where \begin{multline*}
 		p(\eta) := 4\left(\tfrac{1}{\lambda\alpha} - \tfrac{1}{2}\right)\eta^2 - 2 \left(\tfrac{2L_f}{\alpha} + \left(\tfrac{\lambda}{2} + \tfrac{|1-\lambda\alpha|}{\alpha} \right) L_h + \frac{\lambda \rho_f}{2} \right)\eta  - \lambda\left( \rho_f^2 + L_f^2\left(\tfrac{1-\alpha}{\alpha}\right) + \tfrac{L_fL_h}{\alpha} \right).
 	\end{multline*}  Observe that $p$ is a convex quadratic function with a unique strictly positive root $\eta^+$, since $\lambda\alpha < 2$. Furthermore, its intercept  is negative as $\alpha \in (0,1)$. Then $\eta = \eta^+$ is the smallest positive value for which $p(\eta) \geq 0$. Hence, $ L^+:=\max\{L_f, \eta^+ +\rho_f\}$ is the smallest $L$ with $L \geq L_f$ such that $C(\hat{\gamma}(L), L) \geq 0$. Furthermore, $\R_{++} \ni \gamma \mapsto C(\gamma, L^+)$ has a strictly negative derivative and thus is strictly decreasing. Therefore, \eqref{eq:DYE-suff-descent} holds with $c = \gamma C(\gamma, L^+) > 0$ for $\gamma \in (0, \min\{\rho_g^{-1}, \rho_f^{-1},\hat{\gamma}(L^+)\})$, since $C(\gamma, L^+) > C(\hat{\gamma}(L^+), L^+) \geq 0$. This completes the proof.
 \end{proof}
 
 \begin{remark} \label{r:stepsize-bound}
 	For ease of presentation, the bounds for $\gamma >0$ and the value of the constant $c>0$ for which the sufficient decrease condition in Theorem~\ref{t:DYE-suff-descent} holds are not explicitly stated. However, they can be easily obtained from the proof: with the notation therein, if $L_f - \rho_f > \Delta$, then $\gamma \in (0, \min\{\rho_g^{-1}, \rho_f^{-1},\hat{\gamma}^+\})$ and $c = \min\{\gamma C(\gamma,L_f), \hat{C}_1(\gamma)\}$, otherwise, $\gamma \in (0, \min\{\rho_g^{-1}, \rho_f^{-1},\hat{\gamma}(L^+)\})$ and $c = \gamma C(\gamma, L^+)$.
 \end{remark}

As we  have shown that the damped DY method behaves as a method of descent for the DYE, we are now ready to prove (subsequential) convergence leveraging Lemma~\ref{l:Opial}.

\begin{corollary}[Subsequential convergence of the damped DY sequence] \label{c:subseq-conv} Let Assumption~\ref{a:ass1} hold, such that $\varphi$ in \eqref{DY-problem} is  level-bounded.  Let $\gamma \in (0, \min\{\rho_g^{-1}, (L_f+L_h)^{-1}\})$ satisfy the bound in Remark~\ref{r:stepsize-bound}. Let  $\alpha \in (0,1)$ be a damping parameter, and $\lambda >0$ a relaxation parameter such that $\lambda\alpha < 2$.  Then, ${\rm DY}_{\gamma,\alpha}$ is asymptotically regular,  the damped shadow DY sequence $(x^k)_{k\geq1}$ of  \eqref{DY-damped-z}  is bounded and any of its cluster points is a stationary point of $\varphi$. Furthermore, $\lim \DYE(x^k)$ is a stationary value.
 \end{corollary}

 \begin{proof}

 From Theorem~\ref{t:DYE-suff-descent}, $(\DYE(x^k))_{k\geq1}$ is  nonincreasing. Hence, it follows that $(x^k)_{k\geq1} \subseteq {\rm lev}_{\DYE}(\DYE(x^0))$, so that $(x^k)_{k\geq1}$ is a bounded sequence due to Lemma~\ref{l:envelope-prop}(iii) and the assumption that $\varphi$ is level-bounded. Therefore, the set of cluster points of $(x^k)_{k\geq1}$  is nonempty. Furthermore, in view of Lemma~\ref{l:envelope-prop}(ii), $\DYE$ is bounded below, and thus so is the sequence $(\DYE(x^k))_{k\geq1}$. Hence, there exists $\varphi^\star \geq  \inf \varphi$ such that $\DYE(x^k) \to \varphi^\star$. From Theorem~\ref{t:DYE-suff-descent} again, we conclude then that $x^{k+1} - x^k \to 0$, so that ${\rm DY}_{\gamma,\alpha}$ is asymptotically regular. Lemma~\ref{l:Opial} implies that any cluster point of $(x^k)_{k\geq1}$  is a stationary point of $\varphi$. This proves the first part of the statement. Next, let $x^\star$ be an arbitrary cluster point of $(x^k)_{k\geq1}$. From  \eqref{eq:damped-DY-operator} and \eqref{eq:damped-shadow-DY}, we also have $T_\gamma(x^k)-x^k\to 0$, and since $\nabla f$ is Lipschitz continuous, then $\nabla f(T_\gamma(x^k))-\nabla f(x^k)\to 0$. Hence \eqref{eq:x-Y} implies $x^k - Y_\gamma(x^k) \to 0$, and thus $Y_\gamma(x^k) \to x^\star$. Moreover, due to \cite[eq. (8)]{liu2019envelope}, for some constant $K(\gamma) > 0$, there holds \begin{equation*}
     \varphi(Y_\gamma(x^k))\leq \DYE(x^k) - K(\gamma)\|Y_\gamma(x^k) - x^k\|^2.
 \end{equation*} Taking the limit, as $\varphi$ is lower semicontinuous, it follows that \begin{equation*}
     \varphi(x^\star) \leq \liminf \varphi(Y_\gamma(x^k)) \leq \limsup \varphi(Y_\gamma(x^k)) \leq \lim \DYE(x^k) = \varphi(x^\star),
 \end{equation*} where the last equality follows from $\varphi^\star = \varphi(x^\star)$, since   continuity of $\DYE$ and Proposition~\ref{DY:critical-characterization-I} imply $\varphi^\star = \DYE(x^*) =  \varphi(x^\star)$. This completes the proof.
 \end{proof}

Under a standard assumption in nonconvex optimisation,  the \emph{Kurdyka-{\L}ojasiewicz} (K{\L}) inequality (see, e.g., \cite{atenas2023unified,attouch2013convergence}), global convergence of the whole damped shadow DY sequence to a stationary point of $\DYE$ can be shown using \cite[Theorem 2.9]{attouch2013convergence}. We formalise this claim in the result below. For completeness, we first recall the K{\L} inequality. A proper lower semicontinuous function $\phi: \R^d \to \R \cup \{+\infty\}$ satisfies the K{\L} inequality at $\bar{x} \in \dom(\partial \phi)$ if
    there exist $\eta \in (0,+\infty]$, a neighbourhood $U$ of $\bar{x}$, and a continuous
    concave function $\psi: [0,\eta) \to [0,+\infty)$, continuously differentiable on
    $(0,\eta)$, with $\psi(0)=0$ and $\psi'>0$ on $(0,\eta)$, such that
    \begin{equation} \label{eq:KL}
        \psi'\big(\phi(x) - \phi(\bar{x})\big) \, {\rm dist}\big(0, \partial \phi(x)\big) \geq 1
        \qquad \text{for all } x \in U \text{ with } \phi(\bar{x}) < \phi(x) < \phi(\bar{x}) + \eta,
    \end{equation} where ${\rm dist}(0, \partial \phi(x)) := \inf\{ \|s\| : s \in \partial \phi(x)\}$.
    We say $\phi$ is a \emph{K{\L} function} if it satisfies the K{\L} property at every point of
    $\dom(\partial \phi)$.

\begin{theorem}[Global convergence of the damped shadow DY sequence] \label{t:global-conv}
    Let Assumption~\ref{a:ass1} hold, such that
    $\varphi$ in \eqref{DY-problem} is level-bounded. Suppose in addition that $f$ and $h$ are $C^2$ functions with Lipschitz continuous Hessians, and that $\DYE$ is a K{\L} function. Let $\alpha \in (0,1)$ be a damping parameter and $\lambda >0$ a relaxation parameter such
    that $\lambda \alpha < 2$, and let
    $\gamma \in \big(0, \min\{\rho_g^{-1},  (L_f+L_h)^{-1}\}\big)$ also satisfy the
    bound in Remark~\ref{r:stepsize-bound}. Then the damped shadow DY sequence
    $(x^k)_{k\geq1}$ generated by \eqref{DY-damped-z}  converges to a stationary point $x^\star \in \R^d$ of
    $\varphi = f+g+h$.
\end{theorem}

\begin{proof}
    Since $g$ is $\rho_g$-weakly convex, 
    by Lemma~\ref{l:DYE-subdiff}, $\DYE$ is continuously
    differentiable with $\nabla \DYE(x) = \gamma^{-1}\big(I - \gamma \nabla^2 (f+h)(x)\big)\big(x - Y_\gamma(x)\big)$. By Corollary~\ref{c:subseq-conv}, the sequence $(x^k)_{k\geq1}$ is bounded and its set of
    cluster points $\Omega$ is nonempty. 

    The proof amounts to verifying \cite[(H1)--(H3)]{attouch2013convergence}. First, \cite[(H1)]{attouch2013convergence} is the sufficient descent condition in \eqref{eq:DYE-suff-descent} with $a:= c/\gamma$. Condition \cite[(H2)]{attouch2013convergence} is a subgradient estimate we proceed to deduce.  In view of boundedness of $(x^k)_{k\geq1}$, Lemma~\ref{l:DYE-subdiff}, \eqref{eq:x-Y}, and \eqref{eq:x-T}, there exists $M>0$, such that for all $k\geq1$, there holds \begin{align*}
         \|\nabla \DYE(x^k)\| \leq & \frac{1}{\gamma}\|I - \gamma \nabla^2(f+h)(x^k)\|_{\rm op} \|x^k - Y_\gamma(x^k)\| \\
         \leq & \frac{M}{\gamma}\left(\frac{1}{\lambda\alpha}\|x^k-x^{k+1}\| + \frac{\gamma L_f}{\lambda}\|x^k-T_\gamma(x^k)\|\right)=\frac{M}{\lambda\alpha\gamma}(1+L_f\gamma)\|x^k-x^{k+1}\|.
     \end{align*} Moreover, since $\nabla^2(f+h)$ is bounded and Lipschitz continuous, and $Y_\gamma$ is Lipschitz continuous from Lemma~\ref{DY-continuity-prop}, it follows from Lemma~\ref{l:DYE-subdiff} that $\nabla\DYE$ is Lipschitz continuous on the (bounded) closed convex hull of $(x^k)_{k\geq1}$ with constant $L_\nabla$. Hence
     $\|\nabla\DYE(x^{k+1})\| \leq \|\nabla\DYE(x^{k})\| + L_\nabla\|x^{k+1}-x^k\|$,
     and thus \cite[(H2)]{attouch2013convergence} holds with constant $b=\frac{M}{\lambda\alpha\gamma}(1+L_f\gamma)+L_\nabla$. The last condition \cite[(H3)]{attouch2013convergence} holds because $(x^k)_{k\geq1}$ has a nonempty set of cluster points and $\DYE$ is locally Lipschitz continuous (Lemma~\ref{l:envelope-prop}(i)). In this manner, \cite[Theorem 2.9]{attouch2013convergence} implies that $(x^k)_{k\geq1}$ converges to a stationary point $x^\star$ of $\DYE$.  Since $I - \gamma \nabla^2(f+h)(x^\star)$ is invertible for $\gamma \in (0, (L_f+L_h)^{-1})$, Lemma~\ref{l:DYE-subdiff} yields $x^\star = Y_\gamma(x^\star)$.  Proposition~\ref{DY:critical-characterization-I} thus implies that $x^\star$ is a stationary point of $\varphi$.
\end{proof}

\begin{remark} \label{r:conv} A few comments on the  assumptions of Theorem~\ref{t:global-conv} are in order.

\begin{itemize}
    \item[(a)] The K{\L} assumption in Theorem~\ref{t:global-conv} holds, for instance, whenever $f$, $h$ and
    $g$ are semialgebraic. In that case, $\DYE$ is semialgebraic as the partial minimisation in
    \eqref{eq:DYE-Q} of a semialgebraic function. This is consistent with the standing assumptions of
    Section~\ref{s:saddle-avoidance} below.

    \item[(b)] The assumption on the Hessians of $f$ and $h$
    can be weakened if we follow an approach similar to
    \cite[Section 4.1]{atenas2025understanding}.

    \item[(c)] When the desingularising function in \eqref{eq:KL} can be taken of the form
    $\psi(s) = C s^{1-\theta}$ for some $C>0$ and some K{\L} exponent $\theta \in [0,1)$, rates of
    convergence for $(x^k)_{k\geq1}$ follow similarly as in \cite[Theorem 3.4]{attouch2010proximal}.
\end{itemize}

  
\end{remark}

 \section{Saddle Point Avoidance of Shadow Sequences} \label{s:saddle-avoidance}

We now proceed to show that the damped shadow DY sequence generated via \eqref{DY-damped-z} avoids strict saddle points with probability one. This property 
will follow, in part, from the analogous results for the FB method, in view of \eqref{DY-FB-identity-2}. 
The next proposition summarises the results we  need, retrieved from \cite[Theorem 4.1]{davis2022proximal} and its proof.

\begin{proposition}[Forward-backward map properties] \label{prop:prox-gradient}
Consider the optimisation problem \eqref{DY-problem}, and let $\bar{x}$ be a stationary point of $\varphi = f +g +h$. Let Assumption~\ref{a:ass1} hold. In addition, suppose that $f$ and $h$ are $C^2$ functions, and that $g$ admits a $C^2$-smooth active manifold $\mathcal{M}$ at  $\bar{x} $ with local defining equations $G = 0$. Then, for any $\gamma \in (0,  \rho_g^{-1})$,
     \begin{enumerate}
     \item[(i)] $Y_{\gamma}(x) \in \mathcal{M}$  for all $x$ near $\bar{x}$,
     \item[(ii)] if $W$ denotes the orthogonal projection onto $\mathcal{T}_\M(\bar{x})$, then the linear map $\overline{H}_{yy} := WH_{yy}W$, where  $H_{yy} := \nabla^2 g(\bar x) + \gamma^{-1}I + \sum_i\bar{\lambda}_i \nabla^2 G_i(\bar x)$, is (symmetric) positive definite, and $\bar{\lambda}$ is a multiplier vector for which $0 = \nabla g (\bar x) + \nabla (f+h) (\bar x) + \sum_i\bar{\lambda}_i \nabla G_i(\bar x) $,
     \item[(iii)] $Y_{\gamma}(\cdot)$ is a $C^1$ map around $\bar x$, such that for all $u \in \mathcal{T}_{\mathcal{M}}(\bar x)$, $\nabla Y_{\gamma} (\bar{x})u = - \overline{H}_{yy}^{-1}\overline{H}_{xy}u$, where $\overline{H}_{xy}:= W H_{xy}W$, with  $ H_{xy} := \nabla^2 (f+h)(\bar x) - \gamma^{-1}I$, is symmetric,
     \item[(iv)] if, in addition, $\bar x$ is a strict saddle point, then $\overline{H}_{xy} +  \overline{H}_{yy}$ has a strictly negative eigenvalue.
 \end{enumerate}

\end{proposition}

\begin{proof}
    Apply \cite[Theorem 4.1]{davis2022proximal} to the $(L_f+L_h)$-smooth function $f+h$ and the $\rho_g$-weakly convex function $g$.
\end{proof}

\subsection{Strict Saddle Point Avoidance: Davis--Yin Splitting}

In this section, we show that the damped DY method \eqref{DY-damped-z} converges to local minimisers almost surely, whenever the strict saddle property holds.  We first state two preliminary results.

\begin{lemma}[{\cite[Lemma 2.14]{davis2022proximal}}] \label{l:damped}
    Consider a map $T: \R^d \to \R^d$ and a damping parameter $\alpha \in (0,1)$. Define the map $R:= (1-\alpha) I + \alpha T $. Then\begin{enumerate}
    \item[(i)] the fixed-points of $R$
 and $T$ coincide,
 \item[(ii)] if $T$ is continuous and the sequence $(x^k)_{k\geq1}$, generated via $x^{k+1} = R(x^k)$ for all $k\geq1$, converges to a point $\bar{x}$, then $\bar{x}$ is a fixed-point of $T$,
 \item[(iii)] if $T$ is differentiable and the Jacobian $\nabla T (\bar x)$ has a real eigenvalue strictly greater than one, then $\bar x$ is an unstable fixed-point of $R$,
 \item[(iv)] if the map $I - T$ is Lipschitz continuous with constant $L > 0$, then $R$ is a lipeomorphism for any $\alpha >0$ such that $\alpha L < 1$.
    \end{enumerate}
\end{lemma}

\begin{lemma}[{\cite[Corollary 2.12]{davis2022proximal}}] \label{main_stab_cor}
    Let $Q : \R^d \to \R^d$ be a lipeomorphism 	and let $\mathcal{U}_Q$ consist of all unstable fixed-points $x$  of 
	$Q$ at  which the Jacobian $\nabla Q(x)$  is  invertible. Then the set of initial conditions attracted by such fixed points $$ W : = \left\{x \in \R^d \colon \lim_{k \rightarrow \infty} Q^k(x) \in \mathcal{U}_Q\right\} $$ has zero Lebesgue measure. 
\end{lemma}

\begin{remark} \label{r:needed}
    From Lemma~\ref{l:damped}(iv) and Lemma~\ref{main_stab_cor}, it is clear that we need the damping term in \eqref{DY-damped-z}. From Lemma~\ref{l:damped}(ii), we also observe it is necessary to first know that the damped method converges to a stationary point. An instance when this happens is shown in Theorem~\ref{t:global-conv}.
\end{remark}

The main result of this section will be a consequence of the two lemmas above. First, we prove two technical lemmas we will need in our analysis.

	\begin{lemma} \label{l:invariance-commute}
		Let $A \in \R^{d\times d}$ be symmetric, let $T \subseteq \R^d$ be a subspace, and
		let $W$ denote the orthogonal projection onto $T$. If $A\,T \subseteq T$, then $A$
		and $W$ commute.
	\end{lemma}
	
	\begin{proof}
		Since $A$ is symmetric, $AT \subseteq T$ implies $A\,T^\perp \subseteq T^\perp$. Let $x \in \R^d$ be arbitrary, and decompose it as $x = u + w$, where
		$u := Wx \in T$ and $w := (I-W)x \in T^\perp$. Then $WAx = WAu + WAw = Au = AWx$. 
		As $x \in \R^d$ was arbitrary, the conclusion follows.
\end{proof}

\begin{lemma}\label{l:inertia-inv}
		Let $X$ and $Y$ be symmetric positive definite linear maps. Then $X^{-1} - Y^{-1}$ and $Y - X$ have the same number of negative eigenvalues.
	\end{lemma}
	\begin{proof}
		Set $M := Y^{1/2}X^{-1}Y^{1/2}$ and $N := X^{-1/2}YX^{-1/2}$. Then $Y^{1/2}\big(X^{-1}-Y^{-1}\big)Y^{1/2} = M - I$ and $X^{-1/2}\big(Y-X\big)X^{-1/2} = N - I$. Since $Y^{1/2}$ and $X^{-1/2}$ are invertible, Sylvester's law of inertia \cite[Theorem 4.5.8]{horn2012matrix} implies that
		$X^{-1}-Y^{-1}$ and $Y-X$  have the same number of negative eigenvalues as $M-I$ and $N-I$, respectively. Moreover, $M = Y^{1/2}(X^{-1}Y)Y^{-1/2}$ and $N = X^{1/2}(X^{-1}Y)X^{-1/2}$ are both similar
		to $X^{-1}Y$, so $M$ and $N$ have the same eigenvalues, and hence so do $M-I$ and $N-I$, from where the claim follows.
\end{proof}

We now start by checking that the conditions in Lemma~\ref{l:damped}(iii) are satisfied.

\begin{proposition}[Unstable fixed points of DY iteration operator] \label{DY:unstable} Consider the optimisation problem \eqref{DY-problem} for which Assumption~\ref{a:ass1} holds, and such that $f$ and $h$ are $C^2$ functions. Let $\bar{x}$ be a strict saddle point of $\varphi = f +g + h$ with $C^2$-smooth active manifold $\mathcal{M}$.  
	If \begin{equation} \label{DY:unstable-ass} \nabla^2 f(\bar x)\,\nabla^2 h(\bar x)
                 = \nabla^2 h(\bar x)\,\nabla^2 f(\bar x) \quad\text{and}\quad\nabla^2 f(\bar x)\,{\mathcal{T}}_{\M}(\bar{x}) \subseteq {\mathcal{T}}_{\M}(\bar{x}),\end{equation}  then, for any $\gamma \in (0,  \min\{\rho_g^{-1}, (L_f+L_h)^{-1}\})$,  $\nabla T_{\gamma}(\bar x)$ has a real eigenvalue that is strictly greater than one.
\end{proposition}

\begin{proof}
We first gather a few useful facts about the matrices involved in the analysis. Set $\hat{P} := W\big(I+\gamma \nabla^2f(\bar x)\big)W$, where $W$ denotes the orthogonal projection onto $\mathcal{T}_\M(\bar{x})$. \begin{itemize}
    \item[$\bullet$] From \eqref{DY:unstable-ass}, $\nabla^2f(\bar x)$ leaves ${\mathcal{T}}_{\M}(\bar{x})$ invariant, then 
 $\hat{P}u = \big(I+\gamma \nabla^2f(\bar x)\big)u$ for all $u \in {\mathcal{T}}_{\M}(\bar{x})$.
 \item[$\bullet$] $- \overline{H}_{xy}$ is positive definite on ${\mathcal{T}}_{\M}(\bar{x})$. Indeed, using $\gamma < (L_f+L_h)^{-1}$ and Proposition~\ref{prop:prox-gradient}(iii), for every
$u \in {\mathcal{T}}_{\M}(\bar{x})\setminus\{0\}$,
\begin{equation*}
    \langle (- \overline{H}_{xy})u,u\rangle = \gamma^{-1}\|u\|^2
    - \big\langle \nabla^2(f+h)(\bar x)u,u\big\rangle
    \geq \big(\gamma^{-1}-(L_f+L_h)\big)\|u\|^2 > 0.
\end{equation*}
\item[$\bullet$] The assumption $\gamma L_f<1$ gives that 
$\hat{P}$ is positive definite on ${\mathcal{T}}_{\M}(\bar{x})$, and $\overline{H}_{yy}$ is positive definite on ${\mathcal{T}}_{\M}(\bar{x})$ as well by
Proposition~\ref{prop:prox-gradient}(ii).
\item[$\bullet$] $\overline{H}_{xy}$ and $\hat{P}$ commute on ${\mathcal{T}}_{\M}(\bar{x})$. Indeed, as $\nabla^2 f(\bar x)$ and $\nabla^2 h(\bar x)$ commute, then so do $H_{xy}$ and $I+\gamma \nabla^2 f(\bar x)$. Combined with the fact that $\nabla^2f(\bar x)$ and $W$ commute by Lemma~\ref{l:invariance-commute} and \eqref{DY:unstable-ass},  this implies, for $u \in {\mathcal{T}}_{\M}(\bar{x})$,
\begin{equation*}
\begin{aligned}
    \overline{H}_{xy}\hat{P}u &= WH_{xy}\big(I+\gamma \nabla^2 f(\bar x)\big)u\,= W\big(I+\gamma \nabla^2 f(\bar x)\big)H_{xy}u\,= \big(I+\gamma \nabla^2 f(\bar x)\big)WH_{xy}u \\
    &= \hat{P}\overline{H}_{xy}u,
\end{aligned}
\end{equation*}which proves the claim. Consider any polynomial $p$ interpolating $t \mapsto t^{1/2}$ at the eigenvalues of $-\overline{H}_{xy}$. Then $(-\overline{H}_{xy})^{1/2} = p(-\overline{H}_{xy})$ commutes with $\hat{P}$ \cite[Theorem 7.2.6(b)]{horn2012matrix}.
\end{itemize}

Next, in view of \eqref{DY:T-operator}, continuous differentiability of $\prox_{\gamma f}$, and Proposition~\ref{prop:prox-gradient}(iii), $T_{\gamma}$ is a $C^1$ map locally around $\bar x$. Furthermore, from \eqref{T-OC}, for all $x$ sufficiently close to $\bar x$,\begin{equation} \label{eq:take-derivative}
     \gamma\nabla^2 f\big(T_{\gamma}(x) \big) \nabla T_{\gamma}(x) + \nabla T_{\gamma}(x) = (1-\lambda)I + \lambda \nabla Y_{\gamma}(x) + \gamma\nabla^2 f(x).
\end{equation} Then, for $x = \bar x$, Proposition~\ref{DY:critical-characterization-I} and Proposition~\ref{prop:prox-gradient}(iii) yield, for all $u \in \mathcal{T}_{\mathcal{M}}(\bar x)$, \begin{align*}
       \big( I + \gamma\nabla^2 f(\bar x) \big) \nabla T_{\gamma}(\bar x)u& =  \big((1-\lambda)I + \lambda \nabla Y_{\gamma}(\bar x) + \gamma\nabla^2 f(\bar x)\big)u \\
       & =  (1-\lambda)u - \lambda \overline{H}_{yy}^{-1}\overline{H}_{xy}u + \gamma\nabla^2 f(\bar x)u.
\end{align*}As $I + \gamma\nabla^2 f(\bar x)$ is invertible, $\mu$ is a real eigenvalue of $\nabla T_{\gamma}(\bar x)$ with an
associated eigenvector $u \in \mathcal{T}_{\mathcal{M}}(\bar x)$ if and only if\begin{equation*}
    \begin{array}{rrl}
         &  \nabla T_{\gamma}(\bar x)u& = \mu u\\
        \iff  & \big((1-\lambda)I -\lambda  \overline{H}_{yy}^{-1}\overline{H}_{xy} + \gamma\nabla^2 f(\bar x) \big)u & = \mu\big( I + \gamma\nabla^2 f(\bar x) \big) u \\
        \iff  & \big( \hat{P}-\lambda I -\lambda  \overline{H}_{yy}^{-1}\overline{H}_{xy}  \big)u & = \mu\hat{P} u \\
        \iff  & \big( I -   \overline{H}_{yy}^{-1}(-\overline{H}_{xy})  \big)u & = \dfrac{1-\mu}{\lambda}\hat{P} u \\
    \end{array}
\end{equation*} Substituting
$u = (-\overline{H}_{xy})^{-1/2}v$, premultiplying by $(-\overline{H}_{xy})^{1/2}$, and using
that $(-\overline{H}_{xy})^{1/2}$ commutes with $\hat{P}$, the last identity is equivalent to \begin{equation} \label{eq:symmetric-pencil}
    \big(I - (-\overline{H}_{xy})^{1/2}\overline{H}_{yy}^{-1}(-\overline{H}_{xy})^{1/2}\big)v = \frac{1-\mu}{\lambda}\,\hat{P}v, \quad v = (-\overline{H}_{xy})^{1/2}u \neq 0.
\end{equation} The matrices on both sides of the identity in \eqref{eq:symmetric-pencil} are symmetric.
Accordingly, define the symmetric matrix \begin{equation*}
    H(\mu) := \big(I - (-\overline{H}_{xy})^{1/2}\overline{H}_{yy}^{-1}(-\overline{H}_{xy})^{1/2}\big) + \frac{\mu-1}{\lambda}\,\hat{P}.
\end{equation*} By \eqref{eq:symmetric-pencil}, $\mu$ is a real eigenvalue of
$\nabla T_\gamma(\bar x)$ with eigenvector in $\mathcal{T}_{\mathcal{M}}(\bar x)$ if and only if $H(\mu)$ is singular. Hence, it suffices to show that $H(\mu)$ is singular for some $\mu > 1$.

On the one hand, we have $(-\overline{H}_{xy})^{-1/2}H(\mu = 1)(-\overline{H}_{xy})^{-1/2} =(-\overline{H}_{xy})^{-1}-\overline{H}_{yy}^{-1}$, and thus by Sylvester's law of inertia \cite[Theorem 4.5.8]{horn2012matrix},  $H(\mu = 1)$ has the same number of negative eigenvalues as $(-\overline{H}_{xy})^{-1}-\overline{H}_{yy}^{-1}$, which in turn, by Lemma~\ref{l:inertia-inv} applied to the positive definite maps $-\overline{H}_{xy}$ and $\overline{H}_{yy}$,
has the same number of negative eigenvalues as $\overline{H}_{yy} + \overline{H}_{xy}$. Since $\bar x$ is a strict saddle point,
Proposition~\ref{prop:prox-gradient}(iv) ensures that $\overline{H}_{yy} + \overline{H}_{xy}$ has a strictly negative
eigenvalue.

On the other hand, since $I - (-\overline{H}_{xy})^{1/2}\overline{H}_{yy}^{-1}(-\overline{H}_{xy})^{1/2}$ and $\hat{P}$ are symmetric, Weyl's inequality \cite[Theorem 4.3.1]{horn2012matrix} applies, thus for $\mu>1$,
\begin{equation*}
    \lambda_{\min}\big(H(\mu)\big) \geq \lambda_{\min}\big(I - (-\overline{H}_{xy})^{1/2}\overline{H}_{yy}^{-1}(-\overline{H}_{xy})^{1/2}\big)
    + \frac{\mu-1}{\lambda}\,\lambda_{\min}\big(\hat{P}\big).
\end{equation*} As $\hat{P}$ is positive definite, we have 
$\lambda_{\min}\big(H(\mu)\big) > 0$ for sufficiently large $\mu>1$. Hence, resorting to the continuity of the minimal eigenvalue of $H(\mu)$ as a function of $\mu$, there exists $\mu >1$ such that the minimal eigenvalue of $H(\mu)$ is zero. Therefore, $H(\mu)$ is singular for some $\mu >1$, from where we can conclude.\qedhere

\end{proof}

\begin{remark}[On assumption \eqref{DY:unstable-ass}] 
Observe that assumption \eqref{DY:unstable-ass} holds, for instance,  whenever $\nabla^2 f(\bar x) = c\,I$ for some $c \in \R$. In particular,
it is satisfied by $f = \frac{c}{2}\|\cdot\|_2^2 + \langle
b,\cdot\rangle$ with $c \geq 0$ and $b \in \R^d$, including the Tikhonov/ridge regularisation term in typical applications of interest. 
Moreover, the second hypothesis in \eqref{DY:unstable-ass}  also holds in the
following situations.
\begin{enumerate}
    \item[(a)] If $g$ is a $C^2$ function, then $\M = \R^d$. Only the commutativity property is required. In this fully smooth
    setting, if $h = 0$, then the first hypothesis in \eqref{DY:unstable-ass} is vacuous, and
    Proposition~\ref{DY:unstable} holds without any additional assumption. This is 
    \cite[Theorem~5.4]{liu2019envelope}.
    \item[(b)] Suppose $f(x) = \sum_{i=1}^d f_i(x_i)$ and $g(x) = \sum_{i=1}^d g_i(x_i)$
    are separable (e.g. choose $g$ as the $\ell_1$-norm, the MCP or the SCAD). Then
    $\nabla^2 f(\bar x)$ is a diagonal matrix.
    Moreover, if each $g_i$ is either $C^2$ near $\bar x_i$, or nonsmooth at
    $\bar x_i$ with $-\nabla_{x_i} (f+h)(\bar x) \in \inte \partial g_i(\bar x_i)$, then $\M = \M_1 \times \cdots \times \M_d$ is an active manifold of $g$
    at $\bar x$, with $\M_i = \R$ in the first case, and $\M_i = \{\bar x_i\}$ in the
    second. Consequently $\M = \bar x + \text{span}(\{e_j: j \notin \{i: \M_i = \{\bar{x}_i\} \}\})$, where $\{e_j\}_{j=1}^d$ is the set of canonical vectors in $\R^d$. Then the second property in \eqref{DY:unstable-ass} holds. 
    \item[(c)] As pointed out by an anonymous referee, if $\nabla^2 f(\bar x)$ and $\nabla^2 h(\bar x)$ fail to commute, $\nabla T_\gamma(\bar x)$ may have no real eigenvalues, even for stepsizes allowed in Proposition~\ref{DY:unstable}. Take $f$, $g$, $h$ quadratic with Hessians $F=\begin{pmatrix}-3&-2\\-2&-1\end{pmatrix}$, $G=\begin{pmatrix}-2&2\\2&-2\end{pmatrix}$, $H=\begin{pmatrix}0&0\\0&-4\end{pmatrix}$, so that $FH\neq HF$ and $\bar x=0$ is a strict saddle point. For $\gamma=\dfrac19<\min\{\rho_g^{-1},(L_f+L_h)^{-1}\}$ and $\lambda=\dfrac14$, $\nabla T_\gamma(\bar x)=\dfrac{1}{440}\begin{pmatrix}570&-7\\5&573\end{pmatrix}$, whose eigenvalues are approximately $1.2989\pm 0.0130\,i$. Hence, \eqref{DY:unstable-ass} is required by our proof technique.

\end{enumerate}

\end{remark}

We now state and prove the main result of this section, avoidance of strict saddle points for the damped shadow DY sequence generated by \eqref{DY-damped-z}.

\begin{theorem}[Davis--Yin splitting: global escape] \label{DY:escape}
Let Assumption~\ref{a:ass1} hold. Suppose, in addition, that $f$ and $h$ are $C^2$, that $\varphi = f +g + h$ has the strict saddle property,  and that \eqref{DY:unstable-ass} holds at every strict saddle point $\bar x$ of $\varphi$. Choose a stepsize $\gamma \in (0, \min\{\rho_g^{-1},(L_f+L_h)^{-1}\})$, a relaxation parameter $\lambda >0$, and a damping parameter $\alpha \in (0,1)$ such that $\alpha L_2 < 1$, where \begin{equation*}
    L_2 = (1+ \gamma L_f + \lambda L_1)\max\left\{1,\frac{\gamma L_f}{1 - \gamma L_f}\right\} + \lambda L_1 + \gamma L_f \text{~and~} L_1 = 1 + \left( \dfrac{1 + \gamma (L_f+L_h)}{1 - \gamma\rho_g}\right).
\end{equation*}Consider the damped  shadow DY sequence $(x^k)_{k\geq1}$ generated by \eqref{DY-damped-z}. Then, for almost all initial points $x^0 \in \R^d$, if the limit $x^\star$ of $(x^k)_{k\geq1}$ exists, then $x^\star$  is a local minimiser of $\varphi$.
\end{theorem}

\begin{proof}

  First, in view of  
    \begin{enumerate}
        \item[(i)] Lemma~\ref{l:damped}(i) and (ii) and Proposition~\ref{DY:critical-characterization-I}, the fixed-points of ${\rm DY}_{\gamma,\alpha}$ coincide with the stationary points of $\varphi$, and if $(x^k)_{k\geq1}$ converges to some $x^\star$, then $x^\star$ is a stationary point of $\varphi$,
        \item[(ii)] Lemma~\ref{l:damped}(iii) and Proposition~\ref{DY:unstable}, if $x^\star$ is a strict saddle point of $\varphi$, such that $\varphi$ admits a $C^2$-smooth active manifold $\mathcal{M}$ at $x^\star $, then $x^\star$ is an unstable fixed-point of ${\rm DY}_{\gamma, \alpha}$, 
        \item[(iii)] Lemma~\ref{l:damped}(iv), ${\rm DY}_{\gamma, \alpha}$ is a lipeomorphism, since\begin{equation*}
        \begin{array}{rcl}
            I - T_{\gamma}  & = &  (1-\lambda) I + \gamma \nabla f + \lambda Y_{\gamma} - \prox_{\gamma f} \circ \big( (1-\lambda)I + \gamma \nabla f + \lambda Y_{\gamma} \big) \\
            & & \qquad + \lambda I - \gamma \nabla f - \lambda Y_{\gamma} \\
             & = &  \gamma \nabla f_{\gamma}^M \circ \big( (1-\lambda) I + \gamma \nabla f + \lambda Y_{\gamma}\big) + \lambda I - \gamma \nabla f - \lambda Y_{\gamma} \\
             & = &  \gamma \nabla f_{\gamma}^M \circ \big( I + \gamma \nabla f - \lambda( I - Y_{\gamma})\big) + \lambda ( I - Y_{\gamma}) - \gamma \nabla f.
        \end{array}
    \end{equation*}  Note that $I - Y_{\gamma}$ is Lipschitz continuous with constant $L_1$ in view of Lemma~\ref{DY-continuity-prop}. Then $ I - T_{\gamma}$ is Lipschitz continuous with constant $L_2$, since $\nabla f_{\gamma}^M$ is Lipschitz continuous with constant $\max\{\gamma^{-1},\frac{L_f}{1-\gamma L_f}\}$ and $\nabla f$ is Lipschitz continuous with constant $L_f$. Furthermore, at every point where ${\rm DY}_{\gamma,\alpha}$ is differentiable, its Jacobian $I - \alpha\nabla(I-T_\gamma)$ is invertible, as $\|\alpha\nabla(I-T_\gamma)\|_{\rm op} \leq \alpha L_2 < 1$.
    \end{enumerate}
    In this manner, given $x^0 \in \R^d$, suppose $x^k \to x^\star$. In view of (i), $x^\star$ is a stationary point of $\varphi$. If $x^\star$ is not a local minimiser, then (ii) implies it is an unstable fixed-point of ${\rm DY}_{\gamma,\alpha}$. Using (iii) and applying Lemma~\ref{main_stab_cor} with $Q ={\rm DY}_{\gamma,\alpha}$ yields a contradiction, almost surely. 
\end{proof}

\begin{remark}Global escape (almost surely) of the damped shadow DY sequence can be guaranteed whenever $\gamma >0$ also satisfies the bounds in Remark~\ref{r:stepsize-bound} and Theorem~\ref{t:global-conv}, as in this case  the limit of this sequence is guaranteed to exist. Furthermore, the redefinition of the DYE in \eqref{DY-env-x} omits the explicit use of $\prox_{\gamma f}$ that appears in the relation $\DYE \circ \prox_{\gamma f} = V_\gamma $, where $V_\gamma$ is the merit function in \eqref{DY-z}. Otherwise, the Jacobian of $\prox_{\gamma f}$ would appear in  \eqref{eq:take-derivative} explicitly, and the damping would need to be performed in the $(z^k)_{k \geq 1}$ sequence instead, with an indirect impact on the numerical performance of the shadow sequence.\end{remark}

\subsection{Strict Saddle Point Avoidance: Douglas--Rachford Splitting}

As an application of Theorem~\ref{DY:escape} with $h=0$, we obtain the avoidance of strict saddles for the Douglas--Rachford method.

\begin{corollary}[Douglas--Rachford splitting: global escape] \label{DR:escape}
Let Assumption~\ref{a:ass1} hold with $h = 0$. Suppose, in addition, that $f$ is a $C^2$ function, that $\varphi = f +g $ has the strict saddle property, and that $\nabla^2 f(\bar x)\,{\mathcal{T}}_{\M}(\bar{x}) \subseteq {\mathcal{T}}_{\M}(\bar{x})$ for every strict saddle point $\bar x$ of $\varphi$ with $C^2$-smooth active manifold $\mathcal{M}$. Choose a stepsize $\gamma \in (0, \min\{\rho_g^{-1},L_f^{-1}\})$, a relaxation parameter $\lambda >0$, and a damping parameter $\alpha \in (0,1)$ such that $\alpha L_2 < 1$, where \begin{equation*}
    L_2 =  \gamma L_f + \lambda L_1 + (1+ \gamma L_f + \lambda L_1)\max\left\{1,\frac{\gamma L_f}{1 - \gamma L_f}\right\}, \text{~and~} L_1 = 1 + \left( \dfrac{1 + \gamma L_f}{1 - \gamma\rho_g}\right).
\end{equation*}Consider the damped shadow Douglas--Rachford sequence generated by \eqref{DY-damped-z} with $h=0$. Then, for almost all initial points $x^0 \in \R^d$, if the limit $x^\star$ of $(x^k)_{k\geq1}$ exists, then $x^\star$  is a local minimiser of $\varphi$.
    
\end{corollary}

\begin{remark} A few remarks are in order regarding Theorem~\ref{DY:escape} and Corollary~\ref{DR:escape}.
    \begin{enumerate}
        \item[(a)] Corollary~\ref{DR:escape} extends \cite{liu2019envelope}, where the authors prove a similar result for the fully smooth case (where $f$, $h$ and $g$ are $C^2$ functions) of the Douglas--Rachford method \cite[Theorem 5.4]{liu2019envelope}. Theorem~\ref{DY:escape} also includes as a special case the fully smooth case of the forward-backward method in \cite[Theorem 5.4]{liu2019envelope}. Therein, these two methods are treated separately. Hence, our analysis unifies and extends the latter two. Furthermore, Theorem~\ref{DY:escape} for $f=0$ corresponds to the structured nonsmooth case in \cite[Theorem~4.2]{davis2022proximal}.
        
        \item[(b)] If $h=0$ and $\lambda =2$ in Theorem~\ref{DY:escape}, we obtain strict saddle point avoidance for the Peaceman--Rachford splitting method \cite{peaceman1955numerical,li2017peaceman}. To the best of the author's knowledge, this is the first result of this type for this  method. 
    \end{enumerate} 
\end{remark}

\section{Numerical Experiments} \label{s:numerical}

We perform two sets of experiments. First, since both $f$ and $h$ are assumed smooth, we may always apply the FB method to the two-term objective $(f+h)+g$. The experiment in Section~\ref{s:stepsize-experiment} illustrates the effect of replacing the gradient step on $f$ in the FB method with
$\prox_{\gamma f}$ in DY method. Second, in Section~\ref{s:loc-min}, we compare the performance of the FB and the DY methods for finding local minimisers of a test function with multiple saddle points.

\subsection{DY vs. FB for a Nonconvex Variable Selection  Problem}\label{s:stepsize-experiment}

As mentioned in \cite{atenas2025understanding}, some modern penalties used in statistics fall under the category of weakly convex semialgebraic functions, including the  MCP \cite{Zhang2010NearlyUV}. 
 The elastic net problem \cite{zou2005regularization} is a regularised linear modelling technique for variable selection that can handle collinear features.  Consider a matrix $A \in \R^{m \times d}$ with possibly highly correlated columns, and a vector $b \in \R^m$. The optimisation model of the elastic net is the following regularised residual minimisation problem: \begin{equation} \label{eq:elastic-net}
     \min_{x \in \R^d} \frac{1}{2}\|Ax-b\|_2^2 + \mu\left(\nu P(x) + \frac{1-\nu}{2}\|x\|_2^2 \right),
 \end{equation} where $\mu \geq 0$ controls the level of regularisation, $P$ is a sparsity-inducing penalisation (traditionally, the $\ell_1$-norm), $\|\cdot\|_2^2$ induces grouped selection of correlated predictors, and $\nu \in [0,1]$ controls the weight of each component in the hybrid penalty term. 
 When $P = \|\cdot\|_1$, the model reduces to the LASSO if  $\nu = 1$ \cite{tibshirani1996regression}, while for $\nu = 0$, it becomes the ridge regression \cite{hoerl1970ridge}. 
 For nonconvex penalties like MCP,  values larger than a predetermined threshold are not penalised, and thus sparsity  only applies to variables with ``small values''. This reduces the \emph{shrinking bias} effect of large coefficients known for the $\ell_1$-norm.
 
In problem \eqref{DY-problem}, set $f = \frac{1}{2}\|A\cdot-b\|_2^2$, so that
$\prox_{\gamma f}(u) = (I+\gamma A^\top A)^{-1}(u + \gamma A^\top b)$, $h = \frac{\mu(1-\nu)}{2}\|\cdot\|_2^2$, and $g = \mu\nu\,{\rm MCP}$. We take $m=200$, $d=100$, a matrix $A$ with groups
of highly collinear columns, a sparse vector $x_{\rm true}$, and we rescale $A$ so that
$L_f = \|A\|_2^2 = 10^3$. 
Both methods are initialised
from the same point, and both are stopped when a stationarity residual $\gamma_0^{-1}\|x^k - \prox_{\gamma_0 g}\big(x^k - \gamma_0 \nabla(f+h)(x^k)\big)\|$, computed with a fixed reference
stepsize $\gamma_0$, falls below $10^{-8}$.  
Figure~\ref{fig:stepsize} reports the
number of iterations required as a function of $\gamma$, where markers placed at the top of the plot
indicate that the corresponding method diverged or stalled. Whenever convergent, both methods reach the same stationary points for the same parameter settings.

\begin{figure}[ht!]
    \centering
    \includegraphics[width=0.85\linewidth]{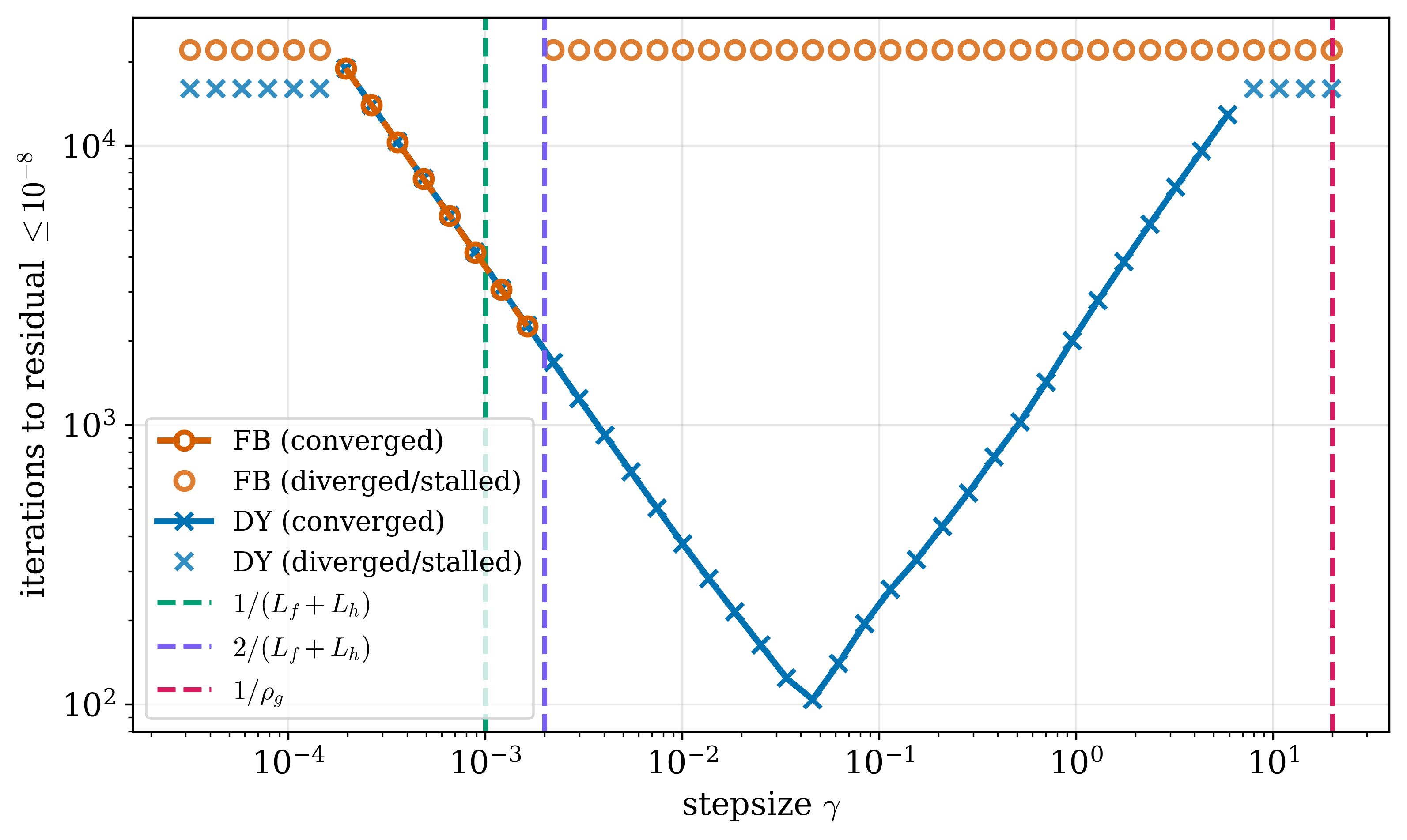}
    \caption{Number of iterations required by the DY and the FB
    methods when solving \eqref{eq:elastic-net} to reach a stationarity residual below $10^{-8}$, as a function of the
    stepsize $\gamma$.}
    \label{fig:stepsize}
\end{figure}

Two features of Figure~\ref{fig:stepsize} are worth emphasising. First, the FB method
is convergent only for $\gamma \lesssim 2 (L_f+L_h)^{-1}$. This restriction is intrinsic to the use of a gradient step. In contrast, the DY method remains convergent
more than three orders of magnitude beyond that threshold. Its (empirical) range of admissible stepsizes is
limited instead by  $\rho_g^{-1}$. Second, the
stepsize at which the DY method performs best lies well inside the region where the
FB method has already diverged. Since one back-substitution of $I+\gamma A^\top A$ to compute $\prox_{\gamma f}$ and one matrix-vector product have comparable cost here, the
same conclusion holds in terms of running time. 
The experiments indicate that the stepsize bounds obtained in this work, while sufficient, are far from necessary. This phenomenon has also been observed in practice for other splitting methods \cite{li2016douglas,alcantara2024four}.  

Since it is difficult to classify the stationary points of problem \eqref{eq:elastic-net},  a meaningful illustration of the saddle point avoidance property is harder to obtain. In the next section we study this property in a test function, capitalising on the findings of Figure~\ref{fig:stepsize}.  

\subsection{Convergence to Local Minimisers: DY vs. FB} \label{s:loc-min}

In this section we demonstrate the saddle point avoidance property on a test function inspired by \cite[Section 4]{aragonartacho2023boosted}. We consider the smooth functions $f,h: \R^2 \to \R$ such that $$f(x)=\frac{\epsilon}{2}\|x\|^2, \quad h(x)=\left(\sum_{j=1}^{2}B_j\,e^{-\|x-p_j\|^2/(2\tau_j^2)}-\sum_{i=1}^{6}A_i\,e^{-\|x-c_i\|^2/(2\sigma_i^2)}\right) ,$$ for $\epsilon, B_1, B_2, \tau_1, \tau_2 >0$, $p_1, p_2 \in \R^2$, $A_i, \sigma_i >0$, and $c_i \in \R^2$, $i \in \{1, \ldots, 6\}$. The function $h$ has six Gaussian valleys of different depths located at $c_i \in \R^2$, and two Gaussian hills located at $p_j \in \R^2$. Every term has a bounded Hessian, so $\nabla (f+h)$ is globally Lipschitz with
$L_f + L_h\ \leq\ \epsilon+\sum_{j=1}^2 {B_j}{\tau_j^{-2}}+\sum_{i=1}^6 {A_i}{\sigma_i^{-2}}$,
using $\|\nabla^2 e^{-\|\cdot-c\|^2/2\sigma^2}\|_{\rm op}\le \sigma^{-2}$. Although $f+h$ is not semialgebraic, it is definable in the o-minimal structure $\mathbb{R}_{\exp}$, so the Kurdyka–{\L}ojasiewicz machinery still applies \cite{bolte2007clarke}. The nonsmooth function $g: \R^2 \to \R$ is the componentwise MCP, 
which is proper, lower semicontinuous, semialgebraic and weakly convex. 
The closed form of the MCP proximal operator can be found in \cite[Section 5]{atenas2025understanding}. In Figure~\ref{fig:landscape}, we first show the contours and surface of $\varphi$. Due to the structure of the test function $\varphi$, we can locate its stationary points (see the plot on the left in Figure~\ref{fig:limits-FB}). In particular, there are $19$ stationary points: $7$ local minimisers and $12$ saddle points. 
\begin{figure}[ht!]
    \centering
    \includegraphics[width=\linewidth]{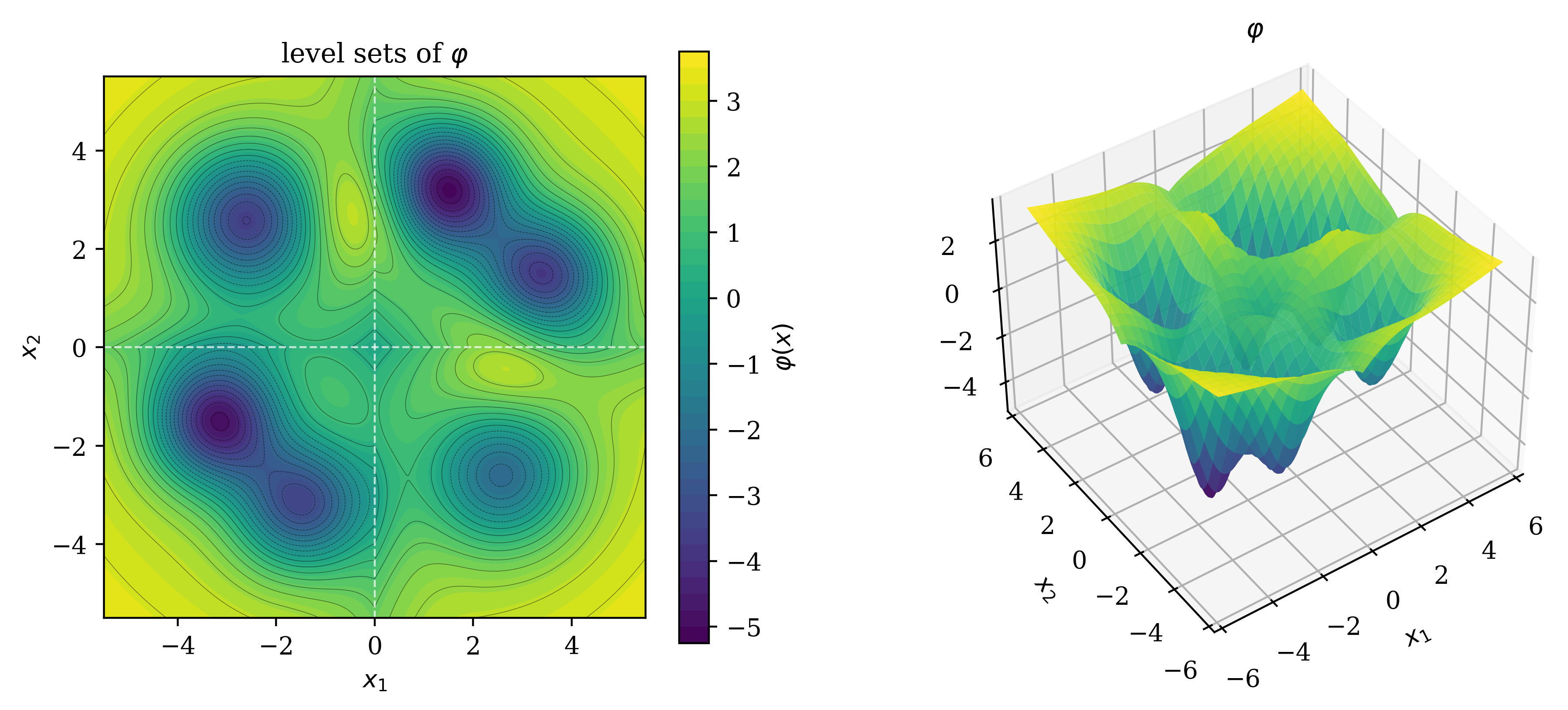}
    \caption{Level sets and graph of the composite function $\varphi$.
    }
    \label{fig:landscape}
\end{figure} 

 We test both the damped DY and the damped FB methods, initialised at $10{,}000$ random points drawn uniformly on $[-5,5]^2$. The methods are stopped when $\|x^{k+1}-x^k\|\leq 10^{-8}$. We set $\lambda = 1$, $\gamma = \frac{0.9}{L_f+L_h}$, $\alpha = \frac{0.9}{L_2}$ for DY, where $L_2$ is the constant in Theorem~\ref{DY:escape}, and $\alpha = \frac{0.9}{L_1}$ for FB. 
 In this setting, every run converged, and none converged to a strict saddle. All limits
matched a local minimiser known to within $1.4\times10^{-10}$, the stationarity
residual $\gamma^{-1}\|x-Y_\gamma(x)\|$ at the limits never exceeding $6.4\times10^{-10}$. The limits reached for both methods coincide, a behaviour we also observed in Section~\ref{s:stepsize-experiment}. Figure~\ref{fig:limits-FB} (left) shows the stationary points, classified as saddle points and local minimisers. It also displays (right) the limits reached in these $10{,}000$ runs,  which in this case all correspond to local minimisers, and their corresponding percentage of attainment.

\begin{figure}[ht!]
    \centering
    \includegraphics[width=\linewidth]{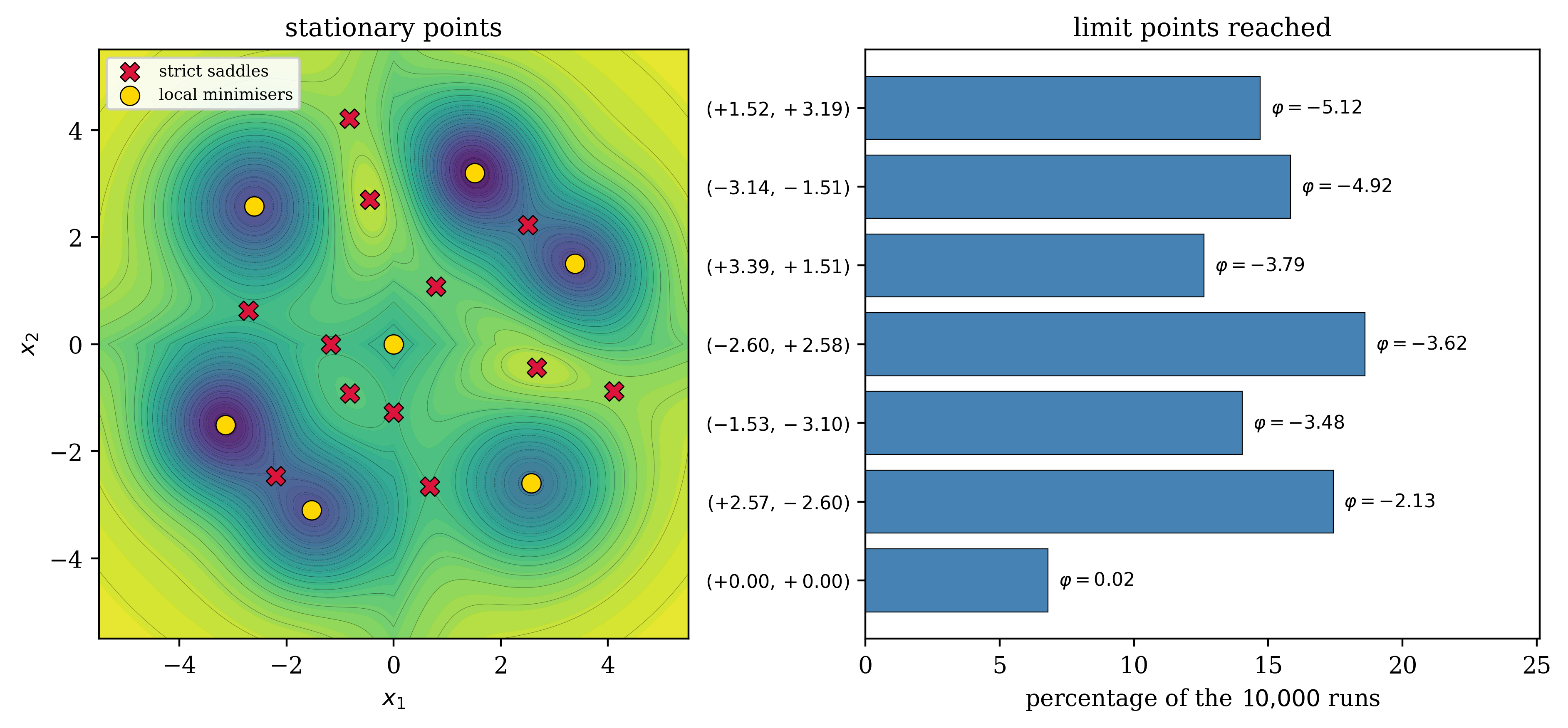}
    \caption{Stationary points of $\varphi$ classified, and their distribution as limit points reached in $10{,}000$ random initialisations of the FB and DY methods. 
    }
    \label{fig:limits-FB}
\end{figure} 
We next compare the performance of the damped FB and damped DY methods. Figure~\ref{fig:boxplot-DYFB} compares the number of iterations required by each method to
meet the stopping criterion, over the $10{,}000$ common initial points. With the parameters
prescribed by the theory, 
the discrepancy is a
consequence of the damping parameter  $\alpha$ each method is allowed. 
For  the FB method,  the admissible $\alpha$ is up to $0.334$, and only up to $0.128$ for
DY, as the constant $L_2$ in Theorem~\ref{DY:escape} accounts for the additional
application of $\prox_{\gamma f}$. When damped DY is run with the
larger value $\alpha = 0.334$, which violates the bound $\alpha L_2 <1$,
its iteration counts for convergence become comparable to those of
the FB method, 
still converging to a local minimiser in every instance. This suggests that the restriction
on $\alpha$ in Theorem~\ref{DY:escape} is conservative, and that the practical cost of the
DY splitting is close to that of the FB method.

\begin{figure}[ht!]
    \centering
    \includegraphics[width=0.75\linewidth]{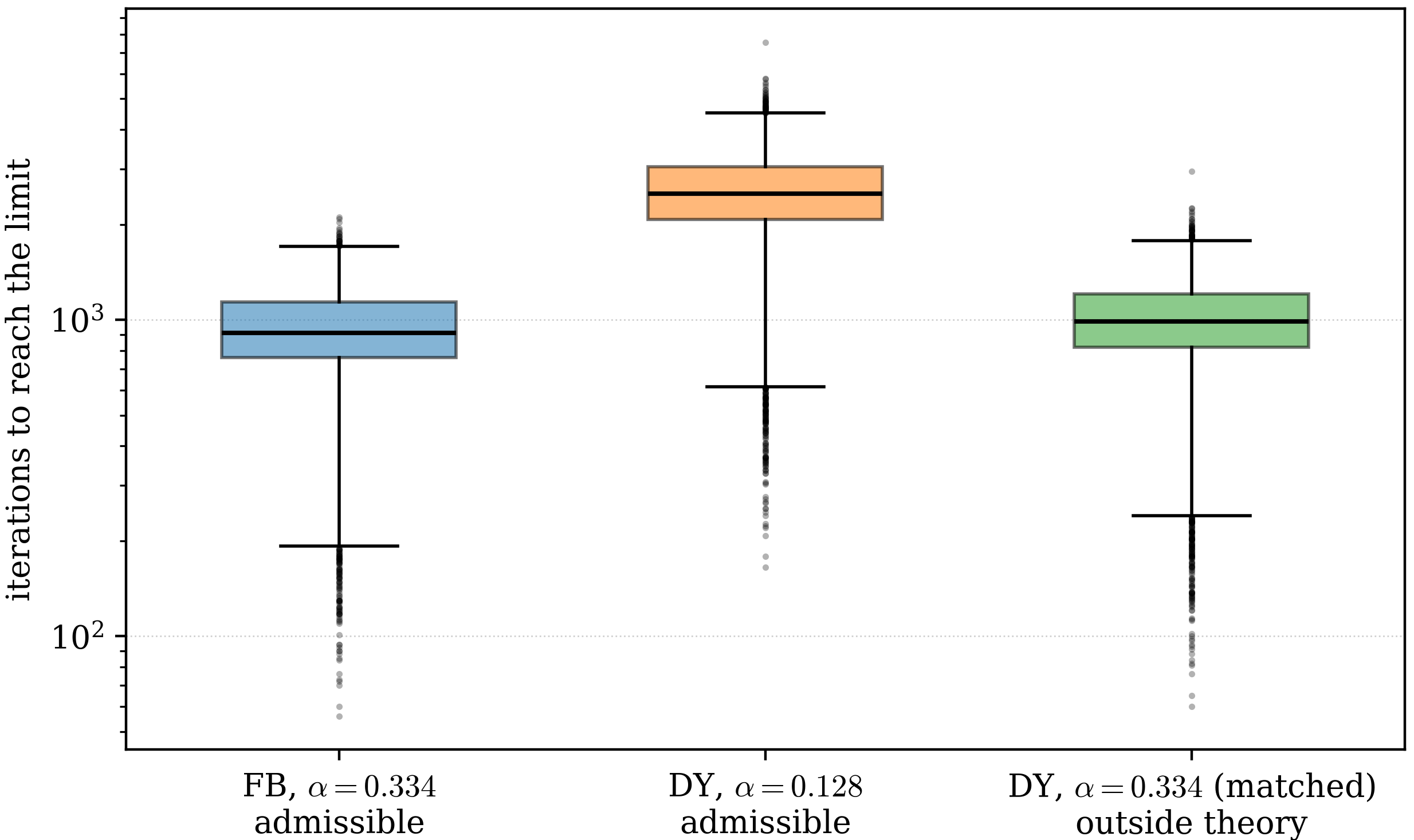}
    \caption{Boxplot of the number of iterations needed for the damped FB and DY methods (first and second box) to reach a limit in $10{,}000$ randomly initialised instances. The third box corresponds to the run of damped DY matching the value of the damping parameter used for FB.
    }
    \label{fig:boxplot-DYFB}
\end{figure}
We conclude our comparison with runs over a range of stepsizes and damping 
parameters that goes well beyond the bounds guaranteed by Theorem~\ref{DY:escape}, the only
restriction being $\gamma < \rho_g^{-1}$. The left panel in Figure~\ref{fig:sweep-curves} shows that, for each value of
$\alpha$, the damped DY method reaches a limit (local minimiser) from every initialisation on a
range of stepsizes roughly twice as large as the one for which the damped FB method
does. At $\alpha = 0.6$, for instance, all runs of the former converge for $\gamma \le 0.45$,
while the latter already fails for $14\%$ of the initialisations at $\gamma = 0.3$ and from more
than half of them at $\gamma = 0.45$. This is consistent with what is known for the undamped
methods  (Section~\ref{s:stepsize-experiment}), where the implicit step on $f$ enlarges the range of admissible stepsizes. 
The right panel in Figure~\ref{fig:sweep-curves} shows that 
when both methods
are reliable, the damped FB method is the faster of the two in terms of iteration counts,
by a factor of about two at each damping parameter. Hence,  the benefit of damped DY  is  robustness with respect to the choice of parameters. The damped DY
method remains convergent on a range of stepsizes where
its damped FB counterpart degenerates.

\begin{figure}[ht!]
	\centering
	\includegraphics[width=\linewidth]{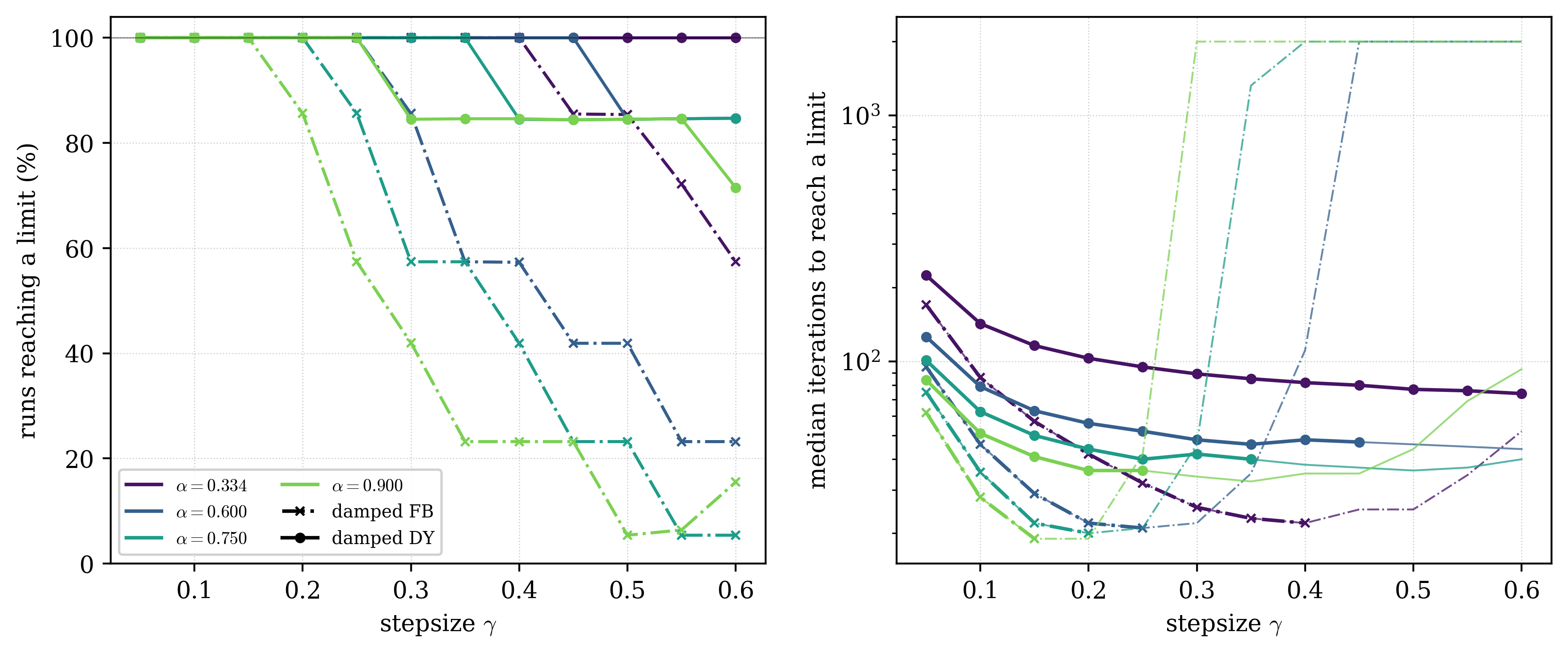}
	\caption{Behaviour of the damped FB (dashed) and damped DY (solid) methods
	as functions of the stepsize $\gamma$, for four damping parameters $\alpha$ and $\lambda=1$, over the
	same $10{,}000$ random initialisations. Left: percentage of runs that reach a local minimiser of
	$\varphi$. Right: median number of iterations needed to do so, drawn in bold on the range of
	stepsizes for which every run reaches a limit, and faded elsewhere.}
	\label{fig:sweep-curves}
\end{figure}

\section{Conclusion} \label{s:conclusion}

In this work, we used the Davis--Yin envelope to explain the asymptotic behaviour of the Davis--Yin splitting method when solving nonconvex optimisation problems. In particular, we showed convergence with probability one to local minimisers of a damped version of the algorithm.

Future research directions will investigate the theoretical justification of the numerical results observed in Section~\ref{s:numerical}, namely, the robustness of the DY method compared to the FB method, even in settings where the latter is applicable. Another research direction is to explore relaxing the assumption of having two of the functions in problem \eqref{DY-problem} smooth to assuming, for instance, that $f$ admits a $C^2$-smooth active manifold. This would allow treating constraints directly via indicator functions, instead of via their Moreau envelopes (the squared distance to the set, which is a smooth function), so that feasibility can be guaranteed in each iteration. This would also expand the applicability of multioperator splitting methods in nonconvex optimisation, as the expected generalisation given by the current theory is that problems of the form $\min \sum_{i=1}^{n-1}f_i + g$ can only be solved via splitting methods if $f_1, \dots, f_{n-1}$ are smooth and $g$ is nonconvex.

\medskip

\noindent\textbf{Funding.} The author was supported in part by Australian Research Council grant DP230101749, and by Centro de Modelamiento Matem\'{a}tico (CMM), ACE210010 and FB210005, BASAL funds for center of excellence from ANID-Chile.



\bibliographystyle{siam}
\bibliography{biblio-R2}
\end{document}